\documentclass[12pt]{article}
\usepackage{graphicx}
\pdfoutput=1
\usepackage[a4paper, total={7in, 10in}]{geometry}
\usepackage[square,sort,comma,numbers]{natbib}
\usepackage{tikz}
\usetikzlibrary{arrows.meta}
\usepackage{amsthm}
\usepackage{amssymb}
\usepackage{amsmath}
\usepackage{pdfpages}
\usepackage{booktabs}
\usepackage{multirow}
\usepackage{graphics}
\usepackage{stmaryrd}
\usepackage{mathtools}
\usepackage[normalem]{ulem}
\usepackage{subcaption}

\usepackage[bookmarksnumbered=true]{hyperref} 
\hypersetup{
     colorlinks = true,
     linkcolor = blue,
     anchorcolor = blue,
     citecolor = teal,
     filecolor = blue,
     urlcolor = blue
     }

\theoremstyle{plain}

\newtheorem{theorem}{Theorem}[section]
\newtheorem*{theoremIntro}{Theorem}
\newtheorem{conjecture}[theorem]{Conjecture}
\newtheorem{proposition}[theorem]{Proposition}
\newtheorem{corollary}[theorem]{Corollary}
\newtheorem{lemma}[theorem]{Lemma}

\theoremstyle{definition}

\newtheorem{definition}[theorem]{Definition}
\newtheorem{example}[theorem]{Example}
\newtheorem{question}[theorem]{Question}

\theoremstyle{remark}

\newtheorem{remark}[theorem]{Remark}

\tikzstyle{vertex}=[circle, fill, draw, inner sep=0pt, minimum size=6pt]
\newcommand{\vertex}{\node[vertex]}

\tikzstyle{smallvertex}=[circle, fill=white, draw, inner sep=0pt, minimum size=6pt]
\newcommand{\smallvertex}{\node[smallvertex]}

\DeclarePairedDelimiter\abs{\lvert}{\rvert}

\DeclarePairedDelimiter{\set}{\{}{\}}

\newcommand{\RR}{\mathbb R}
\newcommand{\ZZ}{\mathbb Z}

\newcommand{\QQ}{\mathbb Q}

\newcommand{\CC}{\mathbb C}

\newcommand{\MD}{\mathcal D}

\newcommand{\conv}{{\rm Conv}}

\usepackage{xcolor}

\DeclareRobustCommand{\Chi}{{\mathpalette\irchi\relax}}
\newcommand{\irchi}[2]{\raisebox{\depth}{$#1\chi$}}
\DeclareRobustCommand{\stirling}{\genfrac\{\}{0pt}{}}

\DeclareMathOperator{\diag}{diag}

\DeclareMathOperator{\GL}{GL}

\DeclareMathOperator{\res}{Res}

\DeclareMathOperator{\im}{Im}
\DeclareMathOperator{\ind}{ind}

\title{Equivariant Ehrhart Theory of Hypersimplices}
\author{Oliver Clarke and Max K\"olbl}
\date{December 2023}

\begin{document}

\maketitle

\begin{abstract}
    We study the hypersimplex under the action of the symmetric group $S_n$ by coordinate permutation. 
    We prove that its equivariant volume, given by the evaluation of its equivariant $H^*$-series at $1$, is the permutation character of {\it decorated ordered set partitions} under the natural action of $S_n$. This verifies a conjecture of Stapledon for the hypersimplex. 
    To prove this result, we give a formula for the coefficients of the $H^*$-polynomial. Additionally, for the $(2,n)$-hypersimplex, we use this formula to show that trivial character need not appear as a direct summand of a coefficient of the $H^*$-polynomial, which gives a family of counterexamples to a different conjecture of Stapledon.
\end{abstract}

\section{Introduction}

The Ehrhart function $i_P \colon \ZZ_{\ge 0} \rightarrow \ZZ_{\ge 0}$ of a polytope $P \subseteq \RR^n$ counts the number of lattice points $i_P(d) = |dP \cap \ZZ^n|$ in its $d$-th dilation, for $d \ge 1$.
By convention, we take $i_P(0) = 1$.
If the vertices of $P$ are lattice points, then $i_P(d) \in \QQ[d]$ is a polynomial called the \emph{Ehrhart polynomial}.
More generally, if the vertices of $P$ have rational coordinates, then $i_P$ is the \emph{Ehrhart quasi-polynomial}, which is a polynomial whose coefficients are periodic functions from $\ZZ$ to $\QQ$.
We can associate a generating function to $i_P$ which we call the \emph{Ehrhart series} denoted ${\rm ES}_P[t] \in \ZZ[[t]]$.
If $P$ is an $n$-dimensional lattice polytope, then the Ehrhart series takes the form
\[
{\rm ES}_P[t] = \sum_{d \ge 0} i_P(d) t^d = \frac{h^*_P(t)}{(1-t)^{n+1}}
\]
for some polynomial $h^*_P(t) \in \ZZ[t]$ of degree $\leq d$ called the \emph{$h^*$-polynomial} of $P$.
Stanley \cite{stanley1980decompositions} showed that the coefficients of $h^*_P$ are all non-negative.
Understanding the relationship between properties of polytopes and properties of their $h^*$-coefficients is one of the central theme in Ehrhart theory \cite{stapledon2009inequalities, higashitani2010shifted, davis2015ehrhart, hofscheier2018ehrhart, braun2016unimodality}. 

\medskip
Given positive integers $k < n$, the \emph{$(k,n)$-hypersimplex} is defined as
\[
\Delta_{k,n} = \{(x_1, x_2, \dots, x_n) \in \RR^n : 
0 \le x_i \le 1 \text{ for all } 1 \le i \le n,\ x_1 + x_2 + \dots + x_n = k\}.
\]
Equivalently, $\Delta_{k,n}$ is the convex hull of the zero-one vectors in $\RR^n$ with exactly $k$ ones.
Stanley \cite[Chapter~4 Exercise~62(e)(h)]{stanley2011enumerative} gave two open problems about the hypersimplex, which were posed as tricky exercises.
The first was to prove \emph{Ehrhart positivity} (the property that all coefficients of the Ehrhart polynomial are non-negative) for the hypersimplex, which was accomplished by Ferroni \cite{ferroni2021positive} using a computation of Katzman \cite{Katzman2005Hilbert}.
The second open problem was to give a combinatorial interpretation of the coefficients of $h^*_{\Delta_{k,n}}(t)$.
In \cite{Katzman2005Hilbert}, Katzman studied so-called \emph{Veronese type algebras}, which includes the Ehrhart ring of the hypersimplex.
He showed 
\[
i_{\Delta_{k,n}}(d) = \sum_{s = 0}^{k-1} (-1)^s \binom{n}{s} \binom{d(k-s) - s + n -1}{n - 1}
\]
and, from this, derived a formula for the $h^*$-polynomial \cite[Corollary~2.9]{Katzman2005Hilbert}.
In \cite{li2012ehrhart}, Li gave the first result towards a combinatorial interpretation of the coefficients of $h^*_{\Delta_{k,n}}$ by proving two combinatorial formulae for the $h^*$-polynomial of the \emph{half-open} hypersimplex.
To do this, Li gave two different shelling orders of a well-known unimodular triangulation of the hypersimplex.
This unimodular triangulation was originally a consequence of a result by Stanley \cite{stanley1977eulerian} and separately arose in the study of Gr\"obner bases associated to the hypersimplex \cite{sturmfels1996grobner}.
Lam and Postnikov \cite{lam2007alcoved} showed that these two triangulations coincide and gave two equivalent triangulations called the \emph{alcove triangulation} and the \emph{circuit triangulation}.
The number of simplices of a unimodular triangulation of a polytope $P$ is equal to its normalised volume, which equals $h^*_P(1)$.
For $\Delta_{k,n}$, the normalised volume is equal to the Eulerian number $A(n-1, k-1)$ \cite{stanley1977eulerian}, which is the number of permutations of $1, \dots, n-1$ with exactly $k-1$ ascents, a fact that was implicit in work of Laplace \cite[p.~257]{laplace1886oeuvres}.

The \textit{Worpitzky identity} \cite{Worpitzky1883} gives a relationship between Eulerian numbers and binomial coefficients. 
In his thesis \cite{early2016thesis, early2016combinatorics}, Early generalised this identity by viewing its terms as dimensions of certain $\CC[S_n]$-modules, which arise from simplices, hypersimplices, and so-called \textit{$q$-plates}, for $q$ a root of unity. In the special case $q=1$, the plates have a combinatorial description, which inspired Early to conjecture in \cite{early2017conjectures} that the coefficients of the $h^*$-polynomial are given in terms of \emph{decorated ordered set partitions} or \emph{DOSPs}.

\begin{definition}\label{def: dosp}
    A $(k,n)$-DOSP is an ordered partition $(L_1, \dots, L_r)$ of $\{1,2, \dots, n\}$ together with a sequence of positive integers $(\ell_1, \dots, \ell_r)$ such that $\ell_1 + \ell_2 + \dots + \ell_r = k$.
    We write a DOSP as a sequence of pairs $D = ((L_1, \ell_1), \dots, (L_r, \ell_r))$.
    A DOSP is defined up to cyclic permutation of its pairs $(L_i, \ell_i)$.
    So, for example, we have $((L_2, \ell_2), \dots, (L_r, \ell_r), (L_1, \ell_1))$ is equal to $D$.
    We say $D$ is \emph{hypersimplicial} if $|L_i| > \ell_i$ for each $i \in \{1,\dots, r\}$.
\end{definition}

For every DOSP $D$, Early defines the \emph{winding number} $w(D) \in \{0,1, \dots, n-1\}$, see Definition~\ref{def: dist in dosp and winding no}, and conjectured that the $h^*$-polynomial is given by
\[
h^*_{\Delta_{k,n}}(t) = \sum_{D} t^{w(D)}
\]
where the sum is taken over all hypersimplicial $(k,n)$-DOSPs. This conjecture was proved by Kim \cite{kim2020combinatorial} and gives a satisfying answer to Stanley's second problem. We now recount the story of Ehrhart theory for hypersimplices in the equivariant setting.

\medskip

\noindent \textbf{Equivariant Ehrhart theory.} 
The symmetric group $S_n$ acts by coordinate permutation on $\RR^n$ and, for each $\sigma \in S_n$, we have $\sigma(\Delta_{k,n}) = \Delta_{k,n}$.
Thus, $\Delta_{k,n}$ is an $S_n$-invariant polytope.
Moreover, the vertices of the hypersimplex are integral and lie in a single $S_n$-orbit.
Such polytopes are known as \emph{orbit polytopes} \cite[Section~3.4.1]{elia2022techniques}.

Motivated by the study of finite groups acting on non-degenerate toric hypersurfaces \cite{stapledon2011representations}, Stapledon introduced an equivariant generalisation of Ehrhart theory \cite{stapledon2011equivariant}.
We provide a detailed description of the setup in Section~\ref{sec: prelim EEH setup}, which defines the equivariant analogues of the Ehrhart series and $h^*$-polynomial.
Intuitively, if $P$ is a $G$-invariant polytope, then the equivariant analogues are obtained by replacing the value of the Ehrhart function $i_P(d)$ with the permutation representation of $G$ acting on the lattice points of the dilation $dP$.
The equivariant analogue of the $h^*$-polynomial is the equivariant $H^*$-series denoted $H^*(P; G)[t] \in \mathcal R_{\CC}(G)[[t]]$, which is a power series with coefficients in the \emph{representation ring of $G$} over $\CC$.

One of the central open problems in equivariant Ehrhart theory is to determine when the $H^*$-series is a polynomial.
Stapledon \cite[Theorem~7.7]{stapledon2011equivariant} showed that if the toric variety corresponding to $P$ admits a non-degenerate $G$-invariant hypersurface, then the $H^*$-series is a polynomial and, moreover, each of its coefficients is \emph{effective}, i.e., each coefficients is a genuine representation.
See Definition~\ref{def: effective representation}. This result led to the following conjecture.

\begin{conjecture}[{\cite[Conjecture~12.1]{stapledon2011equivariant}}]\label{conj: stapledon effectiveness}
    Suppose $P$ is a $G$-invariant lattice polytope and $Y$ is the corresponding toric variety with torus-invariant line bundle $L$. Then the following are equivalent:
    \begin{enumerate}
        \item[(1)] $(Y, L)$ admits a non-degenerate $G$-invariant hypersurface,
        \item[(2)] The coefficients of the equivariant $H^*$-series are effective,
        \item[(3)] The equivariant $H^*$-series is a polynomial.
    \end{enumerate}
\end{conjecture}

The result $(1) \Rightarrow (2) \Rightarrow (3)$ follows from \cite[Theorem~7.7]{stapledon2011equivariant}.
Unfortunately, in \cite[Theorem~1.2]{elia2022techniques}, the authors report that Stapledon and Santos had discovered a counterexample showing $(2) \nRightarrow (1)$. However, the conjecture $(3) \Rightarrow (2)$, which we refer to as the \emph{effectiveness conjecture}, is open and verified only in a handful of cases.

We give a brief overview of the known cases for the effectiveness conjecture.
Stapledon \cite[Proposition~6.1]{stapledon2011equivariant} showed that the coefficients of the $H^*$-polynomial of the simplex, under any invariant group action, are permutation representations. In the same paper, the effectiveness conjecture is also proved for: all $2$-dimensional polytopes under any invariant action, the zero-one hypercube under the action of coordinate permutations, and centrally symmetry polytopes under the action of $\ZZ/2\ZZ$ generated by the diagonal matrix $\diag(-1, -1, \dots, -1)$.
In \cite{ardila2019permutahedron}, the authors use results from \cite{ardila2021equivariant_volumes} to show that the permutohedron, which is the orbit polytope of the point $(1,2,\dots,n) \in \RR^n$ under the action of $S_n$, satisfies the effectiveness conjecture. 
The effectiveness conjecture is known to hold for the hypersimplex $\Delta_{k,n}$ under the action of $S_n$ because its toric variety admits a non-degenerate $S_n$-invariant hypersurface \cite[Theorem~3.60]{elia2022techniques}.
There are also constructive approaches to the effectiveness conjecture using $G$-invariant subdivisions to compute $H^*$-series \cite{elia2022techniques, stapledon2024invariant_triangulations}. In \cite{clarke2023equivariant}, it is shown that the conjecture is false if $P$ is allowed to be a rational polytope.

Suppose that $P$ is a $G$-invariant polytope. For each $g \in G$, the \emph{fixed polytope} $P_g = \{x \in P : g(x) = x\}$ is the set of points in $P$ fixed by $g$. Recall that $h^*_P(1)$ measures the lattice volume of $P$, in particular, it is a non-negative integer. The equivariant analogue $H^*(P; G)[1]$ is the \textit{equivariant volume} of $P$. For example, see \cite{ardila2021equivariant_volumes} for the permutohedron. In \cite{early2016thesis, early2016combinatorics} Early uses $q$-plates to study generalisations of equivariant volumes of simplices and hypersimplices; in particular \cite[Theorem
9]{early2016combinatorics} gives a relationship between these equivariant volumes. Stapledon conjectured the following for equivariant volumes.

\begin{conjecture}[{\cite[Conjectures~12.2, 12.3]{stapledon2011equivariant}}]\label{conj: stapledon H*[1]}
    Let $P$ be a $G$-invariant lattice polytope. For each $g \in G$, we have that $H^*(P; G)[1](g)$ is a non-negative integer. If the $H^*$-series is effective (and polynomial), then $H^*(P; G)[1]$ is a permutation representation.
\end{conjecture}

Part of the strength of the analogy between $h^*_P$ and $H^*(P; G)$ lies in the fact that the evaluation $H^*(G; P)[1](g)$ is directly related to the lattice volume of $P_g$. Currently, one of the few methods to tackle Conjecture~\ref{conj: stapledon H*[1]} is to compute all coefficients of the $H^*$-series.
For instance, if each coefficient $H^*_i$  of the $H^*$-series is a permutation representation, then $H^*(P; G)[1] = \sum_i H^*_i$ is also a permutation representation. The conjecture has been verified for a few examples, including the simplex and zero-one hypercube \cite{stapledon2011equivariant}, and the permutohedron \cite{ardila2019permutahedron}.

In the case of the hypersimplex \cite[Section~3.3]{elia2022techniques}, the coefficients of the $H^*$-polynomial have been computed under the action of $C_n \le S_n$, a cyclic subgroup of order $n$.

\begin{theorem}[{\cite[Theorem~3.3]{elia2022techniques}}]\label{thm: hypersimplex cyclic action}
The coefficient of $t^i$ in $H^*(\Delta_{k,n}, C_n)[t]$ is the permutation representation of $C_n$ acting on the set of hypersimplicial $(k,n)$-DOSPs with winding number $i$. Hence, Conjecture~\ref{conj: stapledon H*[1]} holds for $\Delta_{k,n}$ under the action of $C_n$.
\end{theorem}

In the above result, the action of $C_n$ on the set of hypersimplicial $(k,n)$-DOSPs arises as the restriction of the natural action of $S_n$ on the set of DOSPs given by
\[
\sigma((L_1, \ell_1), \dots, (L_r, \ell_r)) 
= 
((\sigma(L_1), \ell_1), \dots, (\sigma(L_r), \ell_r)),
\]
for each $\sigma \in S_n$. For each $1 \le i \le r$, observe that $|\sigma(L_i)| = |L_i|$. So, the action of $S_n$ leaves the set of hypersimplicial DOSPs invariant. By restricting the action to a particular cyclic subgroup $C_n \le S_n$, it is straightforward to show that the set hypersimplicial DOSPs with a fixed winding number is invariant under $C_n$. However, the set hypersimplicial DOSPs with a fixed winding number is in general not invariant under the action of $S_n$. For instance, the $(2,4)$-DOSP $D = ((12,1), (34, 1))$ has winding number $w(D) = 1$. But, the transposition $\sigma = (2\ 3) \in S_n$ sends $D$ to the DOSP $\sigma(D) = ((13,1), (24,1))$, which has winding number $w(\sigma(D)) = 2$. Therefore, Theorem~\ref{thm: hypersimplex cyclic action} does not hold, or even make sense, if we replace $C_n$ with $S_n$. The aim of this paper is to understand what we can say about the hypersimplex under the action of the full symmetric group.

\medskip

\noindent \textbf{Overview and results.}
In Section~\ref{sec: prelim and notation}, we fix our notation and explain the setup for equivariant Ehrhart theory. 
In Section~\ref{sec: coeffs of H* poly}, we prove our first main result, Theorem~\ref{thm: H* coefficient in terms of representations}, which gives a formula for the coefficients of the $H^*$-polynomial. In particular, this theorem allows us to give a completely combinatorial proof that the equivariant $H^*$-series is a polynomial. See Corollary~\ref{cor: degree of H* poly}.
In Section~\ref{sec: DOSP thm}, we study hypersimplicial DOSPs under the action of $S_n$ and verify Conjecture~\ref{conj: stapledon H*[1]} as follows.

\begin{theoremIntro}[Theorem~\ref{thm: H*[1] is perm character of DOSPs under Sn}]
    We have that $H^*(\Delta_{k,n}; S_n)[1]$ is equal to the permutation representation of $S_n$ acting on the set of hypersimplicial $(k,n)$-DOSPs. Hence, Conjecture~\ref{conj: stapledon H*[1]} holds for $\Delta_{k,n}$ under the action of $S_n$.    
\end{theoremIntro}

The full statement of Theorem~\ref{thm: H*[1] is perm character of DOSPs under Sn} gives a formula for the character $H^*(\Delta_{k,n}; S_n)[1]$. Evaulating this formula at the identity gives the well-known formula for Eulerian numbers
\[
H^*(\Delta_{k,n}; S_n)[1](id) = \sum_{h = 0}^{k-1}(-1)^h \binom{n}{h}(k-h)^{n-1} = A(n-1, k-1).
\]
Similar to the Eulerian numbers, we show that the values $H^*(\Delta_{k,n}; S_n)[1](\sigma)$ satisfy a certain recurrence relation in Section~\ref{sec: recurrence relation}.

In Section~\ref{sec: k=2}, we apply Theorem~\ref{thm: H* coefficient in terms of representations} to the second hypersimplex $\Delta_{2,n}$. The main result of this section is Theorem~\ref{thm: k=2 H* coeff}, which gives a complete description of the coefficients of the $H^*$-polynomial in simple terms. From this, we observe two facts. Firstly, in Corollary~\ref{cor: k=2 H* is effective}, we have a combinatorial proof that the coefficients of the $H^*$-polynomial are effective. Secondly, in Corollary~\ref{cor: k=2 H*_m is perm rep iff m not 1}, we observe that there is exactly one coefficient of the $H^*$-polynomial that is not a permutation representation of $S_n$. We do this by showing that the coefficient of $t$ in $H^*(\Delta_{2,n}, S_n)[t]$ contains no copies of the trivial representation, which provides a family of counter examples to the following conjecture.

\begin{conjecture}[{\cite[Conjecture~12.4]{stapledon2011equivariant}}]\label{conj: stapledon positive coeff}
    Let $P$ be a $G$-invariant polytope. Assume $H^*(P; G)[t]$ is an effective polynomial. If the coefficient of $t^m$ in $h^*_P(t)$ is positive, then the coefficient of $t^m$ in $H^*(P; G)[t]$ contains the trivial representation with non-zero multiplicity.
\end{conjecture}

\begin{example}\label{example: counterexample to stapledon positive coeff}
    Let $n \ge 4$ and let $H^*_1$ be the coefficient of $t$ in $H^*(\Delta_{2,n}; S_n)(t)$. By Corollary~\ref{cor: k=2 H*_m is perm rep iff m not 1}, the multiplicity of the trivial representation in $H^*_1$ is zero. On the other hand, let $h^*_1$ be the coefficient of $t$ in $h^*_{\Delta_{2,n}}(t)$, which is equal to the number of hypersimplicial $(2,n)$-DOSPs with winding number one. In particular, the hypersimplicial DOSP 
    $
    ((12,1), (34\cdots n,1))
    $
    has winding number one, hence $h^*_1 > 0$. Therefore, the hypersimplex $\Delta_{2,n}$ is a counterexample to Conjecture~\ref{conj: stapledon positive coeff}.
\end{example}

We note that Stapledon has already given a counterexample to Conjecture~\ref{conj: stapledon positive coeff} \cite[Example~4.35]{stapledon2024invariant_triangulations} and asked whether the conjecture holds if we further assume that $P$ exhibits a $G$-invariant lattice triangulation. In the case of the hypersimplex $\Delta_{2,n}$, we observe, at the very end of the paper, that it does not admit an $S_n$-invariant lattice triangulation. 

In Section~\ref{sec: discussion}, we discuss questions and future directions based on the results of the paper. We ask for which subgroups $G \le S_n$, the hypersimplex $\Delta_{k,n}$ admits a $G$-invariant triangulation. We conjecture that the well-known unimodular triangulation of $\Delta_{k,n}$ is invariant under the dihedral group $D_{2n}$ and this in the largest subgroup for which the triangulation is invariant.

\section{Preliminaries and notation}\label{sec: prelim and notation}

In this section, we give the general setup for equivariant Ehrhart theory and fix our notation.

\medskip
\noindent \textbf{General notation.} Let $n$ be a positive integer. We write $[n] := \{1,2,\dots,n\}$ and $S_n$ for the symmetric group on $[n]$. Given a finite set $X$, we write $\binom{X}{n}$ for the collection of $n$-subsets of $X$. Group actions are left actions and group representations are over the complex numbers $\CC$.
Given a group $G$, its \emph{representation ring} $\mathcal R(G)$ is the ring of formal differences of isomorphism classes of representations of $G$. Addition in $\mathcal R(G)$ is given by direct sums and multiplication by tensor products. Since we work with finite groups and representations over $\CC$, by a slight abuse of notation, we identify representations with their characters. A \emph{class function} of $G$ is any function $G \rightarrow \CC$ that is constant on conjugacy classes of $G$. Given an element $g \in G$, we denote by $o(g)$ the order of $g$.

\begin{definition}
    Each permutation $\sigma \in S_n$ is a product of disjoint cycles. The \emph{cycle type} of $\sigma$ is the multi-set $(s_1, s_2, \dots, s_r)$ of cycle lengths. The indices of the disjoint cycles of $\sigma$ are the \emph{cycle sets} $C_1, \dots, C_r$, which partition $[n]$. By convention, we assume that $|C_i| = s_i$ for each $i \in [r]$. For example, the permutation  $\sigma = (1) (2 \ 3) (4 \ 5)$ has cycle type $(1,2,2)$ and cycle sets $C_1 = \{1\}$, $C_2 = \{2,3\}$, $C_3 = \{4,5\}$.
\end{definition}

Suppose that $G$ is a finite group that acts on a finite set $X$. The \emph{permutation representation} of the action is the vector space $V = \CC^{X}$, with basis $\{e_x : x \in X\} \subseteq V$, together with an action of $G$ defined by $g \cdot e_x = e_{g(x)}$ for all $g \in G$ and $x \in X$. The group $G$ acts on $V$ by permutation matrices, hence its character $\Chi$ is given by $\Chi(g) = |\{x \in X : g(x) = x\}|$ for each $g \in G$.

\subsection{Equivariant Ehrhart theory}\label{sec: prelim EEH setup}

The setup of equivariant Ehrhart theory requires two ingredients: a group $G$ acting on a lattice and a $G$-invariant polytope $P$ that lies in a codimension-one subspace. From this, we define the equivariant $H^*$-series denoted $H^*(P; G)$, which encodes the $h^*$-polynomials of all fixed polytopes of $P$. If $G$ is the trivial group, then $H^*(P; G)$ is naturally equivalent to the $h^*$-polynomial of $P$. 

\medskip

\noindent \textbf{Group setup.} Throughout this section, let $G \le \GL_{n+1}(\ZZ)$ be a finite subgroup that naturally acts on the lattice $\ZZ^{n+1}$ and, by extension, the vector space $\RR^{n+1} = \RR \otimes_\ZZ \ZZ^{n+1}$. Assume there exists a $G$-fixed $1$-dimensional subspace $A \subseteq \RR^{n+1}$, that is, for all $a \in A$ and $g \in G$ we have $g(a) = a$. Assume there exists a primitive lattice point $a \in \ZZ^{n+1}$ that spans $A$. Let $A^{\perp} \cong \RR^n$ be the orthogonal space to $A$. Let $N = (A^\perp \cap \ZZ^{n+1}) + \langle a \rangle_\ZZ \subseteq \ZZ^{n+1}$ be the rank $n+1$ sub-lattice generated by the lattice points of $A^\perp$ and $a$ and define $d = [\ZZ^{n+1} \colon N]$ the index of $N$ in $\ZZ^{n+1}$. Fix a positive integer $k$. Let $M$ be the $n$-dimensional affine lattice at height $k$ along $A$, that is $M = (\frac kd a + A^\perp) \cap \ZZ^{n+1}$, and $M_\RR = \frac kd a + A^\perp$ its corresponding affine space. Throughout, we identify $\RR^n$ with $M_\RR$ and we identify $\ZZ^n$ with the lattice $M$.

\begin{example}\label{ex: EEH setup S2}
    Let $n = 1$ and let $G \cong S_2$ be the subgroup of $\GL_2(\ZZ)$ of permutation matrices:
    \[
    G = \left\{
    e := 
    \begin{pmatrix}
        1 & 0 \\
        0 & 1 \\
    \end{pmatrix},\,
    \sigma := 
    \begin{pmatrix}
        0 & 1 \\
        1 & 0
    \end{pmatrix}
    \right\}.
    \]
    Let us compute the $G$-invariant affine subspace of $\RR^2$ at height one ($k = 1$) described above and depicted in Figure~\ref{fig: EEH setup}.

    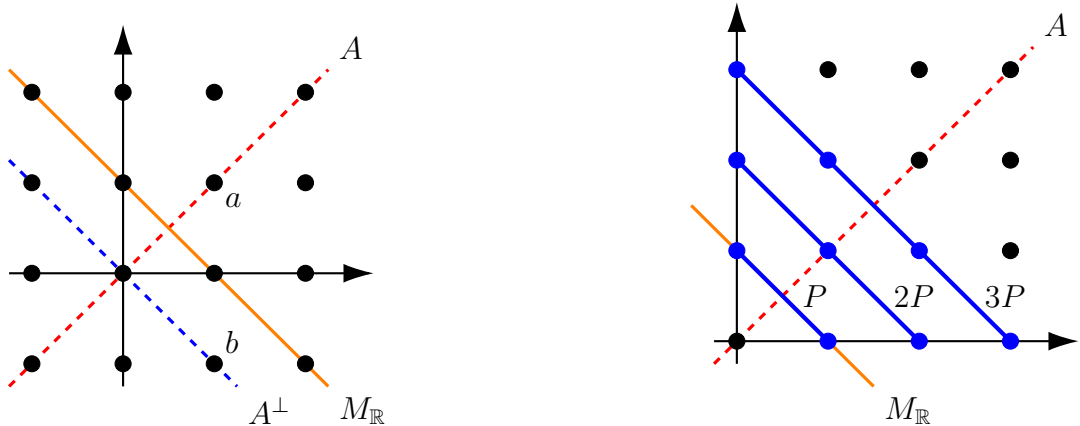
\begin{figure}[t]
    \begin{subfigure}{.5\textwidth}
      \centering
      \begin{tikzpicture}[scale = 0.6]

          \draw[thick, -{Latex[length=4mm, width=2.5mm]}] (-2.5, 0) -- (5.5,0);
          \draw[thick, -{Latex[length=4mm, width=2.5mm]}] (0, -2.5) -- (0, 5.5);

          \draw[very thick, color = red, dashed] (-2.5, -2.5) -- (4.5, 4.5);
          \draw[very thick, color = blue, dashed] (-2.5, 2.5) -- (2.5, -2.5);
          \draw[very thick, color = orange] (-2.5, 4.5) -- (4.5, -2.5);
          
          \vertex [] at (-2,-2) {};
          \vertex [] at (0,-2) {};
          \vertex [] at (2,-2) {};
          \vertex [] at (4,-2) {};
          \vertex [] at (-2,0) {};
          \vertex [] at (0,0) {};
          \vertex [] at (2,0) {};
          \vertex [] at (4,0) {};
          \vertex [] at (-2,2) {};
          \vertex [] at (0,2) {};
          \vertex [] at (2,2) {};
          \vertex [] at (4,2) {};
          \vertex [] at (-2,4) {};
          \vertex [] at (0,4) {};
          \vertex [] at (2,4) {};
          \vertex [] at (4,4) {};

          \node [anchor = north west] at (2,2) {$a$};
          \node [anchor = south west] at (2,-2) {$b$};
          \node [anchor = south west] at (4.5,4.5) {$A$};
          \node [anchor = north west] at (2.5,-2.5) {$A^\perp$};
          \node [anchor = north west] at (4.5,-2.5) {$M_\RR$};
      \end{tikzpicture}
      \caption{Subspaces $A, A^\perp, M_\RR$ in Example~\ref{ex: EEH setup S2}}
      \label{fig: EEH setup}
    \end{subfigure}
    \begin{subfigure}{.5\textwidth}
      \centering
      \begin{tikzpicture}[scale = 0.6]

          \draw[thick, -{Latex[length=4mm, width=2.5mm]}] (-0.5, 0) -- (7.5,0);
          \draw[thick, -{Latex[length=4mm, width=2.5mm]}] (0, -0.5) -- (0, 7.5);
          
          \draw[very thick, color = orange] (-1, 3) -- (3, -1);
          \draw[very thick, color = red, dashed] (-0.5, -0.5) -- (6.5, 6.5);
          \draw[ultra thick, color = blue] (2, 0) -- (0, 2);
          \draw[ultra thick, color = blue] (4, 0) -- (0, 4);
          \draw[ultra thick, color = blue] (6, 0) -- (0, 6);
          
          \vertex [] at (0,0) {};
          \vertex [color = blue] at (2,0) {};
          \vertex [color = blue] at (4,0) {};
          \vertex [color = blue] at (6,0) {};
          \vertex [color = blue] at (0,2) {};
          \vertex [color = blue] at (2,2) {};
          \vertex [color = blue] at (4,2) {};
          \vertex [] at (6,2) {};
          \vertex [color = blue] at (0,4) {};
          \vertex [color = blue] at (2,4) {};
          \vertex [] at (4,4) {};
          \vertex [] at (6,4) {};
          \vertex [color = blue] at (0,6) {};
          \vertex [] at (2,6) {};
          \vertex [] at (4,6) {};
          \vertex [] at (6,6) {};
          
          \node [anchor = north west] at (3,-1) {$M_\RR$};
          \node [anchor = south west] at (6.5,6.5) {$A$};
          \node [anchor = west] at (1.2,1) {$P$};
          \node [anchor = west] at (3.2,1) {$2P$};
          \node [anchor = west] at (5.2,1) {$3P$};
          
      \end{tikzpicture}
      \caption{Dilations of $P$ in Example~\ref{ex: EEH S2 example polytope}}
      \label{fig: EEH example polytope}
    \end{subfigure}
    \caption{Subspaces and polytopes accompanying Examples~\ref{ex: EEH setup S2} and \ref{ex: EEH S2 example polytope}.}
    \label{fig: EEH setup and example polytope}
    \end{figure}
    There is a unique $1$-dimensional $G$-fixed subspace $A = \{(x,x) \in \RR^2 : x \in \RR\}$. We take $a = (1,1)$ a primitive lattice point in $A$. The orthogonal space $A^\perp = \{(x,-x) \in \RR^2 : x \in \RR\}$ is one dimensional with lattice basis $b = (1,-1)$. The sublattice $N \subseteq \ZZ^2$ is freely generated by $\{a,b\}$, so we have
    \[
    d = [\ZZ^2 : N] = \left\lvert
    \det\begin{pmatrix}
        1 & 1 \\
        1 & -1 
    \end{pmatrix}
    \right\rvert
    = 2.
    \]
    We define the affine lattice $M = \frac{1}{2}e + A^\perp$ and identify $\RR^n$ with $M_\RR$.
\end{example}

The next example shows the predominant setup in the rest of the paper.

\begin{example}\label{ex: EEH setup Sn}
    Fix positive integers $k < n$. Fix $G \le \GL_n(\ZZ)$ the set of permutation matrices, which form a subgroup isomorphic to $S_n$. Let $A \subseteq \RR^n$ be the $G$-fixed subspace given by the span of the primitive lattice point $a_1 = (1,1, \dots, 1)$. Let $e_1, \dots, e_n$ be the standard unit vectors of $\RR^n$. The lattice of the orthogonal space $A^\perp \cap \ZZ^n$ is generated by the vectors $\{a_i := e_1 - e_i : 2 \le i \le n\}$. The lattice $N$ generated by $a_1, a_2, \dots, a_n$ has index 
    \[
    d = [\ZZ^n : N] = \left\lvert
    \det
    \begin{pmatrix}
        1 & 1  & 1 & \dots & 1 \\
        1 & -1 & 0 & \dots & 0 \\
        1 & 0  &-1 & \dots & 0 \\
        \vdots & \vdots & \vdots & \ddots & \vdots \\
        1 & 0  & 0 & \dots & -1
    \end{pmatrix}
    \right\rvert
    = n.
    \]
    We identify $\RR^{n-1}$ with the affine space $M_\RR = \frac{k}{n}a_1 +A^\perp \subseteq \RR^n$, which contains the affine lattice $M = \{x \in \ZZ^n : x_1 +\dots + x_n = k\}$.
\end{example}

\noindent \textbf{Polytope setup.} Assume the group setup with $G\le \GL_{n+1}(\ZZ)$ a finite group. Since $\GL_{n+1}(\ZZ) \subseteq \GL_{n+1}(\CC)$ there is a natural $(n+1)$-dimensional representation of $G$ given by the inclusion map $\rho : \GL_{n+1}(\ZZ) \rightarrow \GL_{n+1}(\CC)$. Let $P \subseteq \RR^n$ be a polytope. We say that $P$ is \emph{$G$-invariant} if for every $g \in G$ we have $g(P) = P$. Suppose $P$ is $G$-invariant. The group $G$ acts by permutation on the lattice points of $P$ and the lattice points of each of its dilates $mP \subseteq mM_\RR \subseteq \RR^{n+1}$ for $m \in \ZZ_{\ge 0}$. For each $m \in \ZZ_{\ge 0}$, denote by $\Chi_m$ the character of this permutation representation. By convention, we define $\Chi_0$ to be the trivial character. The \emph{equivariant Ehrhart series} ${\rm EES_P}$  and \emph{equivariant $H^*$-series $H^*(P; G)$} are elements of $\mathcal R(G)[[t]]$ the power series ring with coefficients in $\mathcal R(G)$ defined by
\[
{\rm EES}_P[t] = \sum_{m \ge 0} \Chi_m t^m
= \frac{H^*(P; G)[t]}{\det(I_{n+1} - t \rho)}
\]
where $I_{n+1}$ is the identity matrix of size $n+1$. The function $1/\det(I_{n+1} - t\rho)$ is the class function of $G$ whose evaluation at $g \in G$ is $1/\det(I_{n+1} - t\rho(g))$. We may also think of this as the `character' of a certain representation. Let $V = \CC^{n+1}$ be the $\CC[G]$-module associated to $\rho$, then, by \cite[Lemma~3.1]{stapledon2011equivariant}, we have 
\[
\sum_{m \ge 0} {\rm Sym}^m V t^m = \frac{1}{1 - Vt + \bigwedge^2 Vt^2 - \dots + (-1)^{n+1} \bigwedge^{n+1} V t^{n+1}},
\]
where ${\rm Sym}^m V$ is the $m$-th symmetric power of $V$ and $\bigwedge^m V$ is the $m$-th exterior power of $V$. Moreover, the character of the above representation is $1/\det(I_{n+1} - t\rho)$.

We often write $H^*$ for $H^*(P; G)$ if $P$ and $G$ are clear from context and we write $H^*_m$ for the coefficient of $t^m$ in $H^*$. For each $g \in G$, we have that ${\rm EES}_P[t](g) = \sum_{m \ge 0} \chi_m(g) t^m = {\rm ES}_{P_g}[t]$ is the Ehrhart series of the fixed polytope $P_g = \{x \in P : g(x) = x\}$. If the coefficients of $H^*(P; G)$ are eventually zero, then we call it the \emph{equivariant $H^*$-polynomial} of $P$ with respect to the action of $G$.

\begin{example}\label{ex: EEH S2 example polytope}
Assume the setup in Example~\ref{ex: EEH setup S2} and let $P = \conv\{(1,0), (0,1)\}$ be a line segment. Let us compute the equivariant $H^*$-series of $P$ under the action of $G = \{e, \sigma\}$. The denominator of the equivariant Ehrhart series is given by
\[
\det(I_2 - t\rho(e)) = (1-t)^2 \text{ and }
\det(I_2 - t\rho(\sigma)) = \det\begin{pmatrix}
    1 & -t \\
    -t & 1
\end{pmatrix} = 1-t^2.
\]
The polytope $P$ and its dilations are shown in Figure~\ref{fig: EEH example polytope}. In particular, the $\sigma$-fixed lattice points are shaded. The Ehrhart series of $P_e = P$ and $P_\sigma = \{(1/2, 1/2)\}$ are given by
\begin{align*}
    {\rm ES}_{P_e}[t] &= 1 + 2t + 3t^2 + \dots = 1/(1-t)^2 &= 1/\det(I_2 - t\rho(e)),\\
    {\rm ES}_{P_\sigma}[t] &= 1 + t^2 + t^4 + \dots = 1/(1-t^2) &= 1/\det(I_2 - t\rho(\sigma)). 
\end{align*}
So we have $H^* = \Chi_0$, is the trivial representation.
\end{example}

\begin{definition}\label{def: effective representation}
    Fix a group $G$ and a $G$-invariant polytope $P$. An element of $\mathcal R(G)$ is \emph{effective} if it is a representation, or in other words, it can be expressed as a non-negative integer sum of irreducible representations. We say that the $H^*$-series $H^* = \sum_{m \ge 0} H^*_m t^m$ is \emph{effective} if $H^*_m$ is a effective for each $m \ge 0$.
\end{definition}

\subsection{Hypersimplices}

Fix $0 < k \le n$. We denote by $e_i$ with $i \in [n]$ the standard basis vectors of $\RR^n$. For each subset $I \in \binom{[n]}{k}$, we write $e_I = \sum_{i \in I} e_i$. The \emph{hypersimplex} $\Delta_{k,n} \subseteq \RR^n$ is the convex hull
\[
\Delta_{k,n} = \conv\left\{e_I : I \in \binom{[n]}{k}\right\}
\]
with vertices given by $e_I$ for each $k$-subset $I$ of $[n]$.
The group $S_n$ acts on $\RR^n$ by permutation of its coordinates. The action of $S_n$ satisfies $\sigma(e_I) = e_{\sigma(I)}$, hence the action leaves the polytope $\Delta_{k,n}$ invariant. Following the setup in the previous section and Example~\ref{ex: EEH setup Sn}, the $S_n$-fixed space $A$ is the linear span of $a = (1,1,\dots, 1)$. The hypersimplex lies in the affine space $\RR^{n-1} := \frac kn a +A^\perp$.
We recall, by \cite[Theorem~3.60]{elia2022techniques}, that $H^*(\Delta_{k,n}; S_n)$ is effective and polynomial.

\begin{example}\label{ex: hypersimplex 24}
    Let $S_4$ act by coordinate permutation on $\RR^4$ and consider the $S_4$-invariant hypersimplex $\Delta_{2,4}$, which has vertices
    \begin{align*}
        e_{12} = (1,1,0,0),\, e_{13} = (1,0,1,0),\, e_{14} = (1,0,0,1),\\
        e_{23} = (0,1,1,0),\, e_{24} = (0,1,0,1),\, e_{34} = (0,0,1,1),
    \end{align*}
    as shown in Figure~\ref{fig: hypersimplex 24 whole}.
    The hypersimplex $\Delta_{2,4}$ is a $3$-dimensional polytope that lies in the $S_4$-invariant affine subspace $M_\RR = \{x \in \RR^4 : x_1 + x_2 + x_3 + x_4 = 2\} \cong \RR^3$.
    For each element $\sigma \in S_4$, we obtain a fixed polytope of $\Delta_{2,4}$, which is its intersection with the space of $\sigma$-fixed points. For example, if $\sigma = (1 \ 2)$ then the fixed polytope $(\Delta_{2,4})_{\sigma} = \{x \in \Delta_{2,4} : \sigma(x) = x \}$ is the convex hull of the points
    \[
    e_{12} = (1,1,0,0),\, 
    e_{34} = (0,0,1,1),\, 
    \frac 12(e_{13} + e_{23}) = 
    \left(\frac 12, \frac 12, 1, 0\right),\, 
    \frac 12(e_{14} + e_{24}) = 
    \left(\frac 12, \frac 12, 0, 1\right),
    \]
    see Figure~\ref{fig: hypersimplex 24 slice}. The group $S_4$
    has five conjugacy classes indexed by cycle types. The equivariant $H^*$-series is a quadratic polynomial $H^* = H^*_0 + H^*_1 t + H^*_2 t^2$ with coefficients shown in Table~\ref{tab: coefficients of H* 24}. Theorem~\ref{thm: k=2 H* coeff} shows that $H^*_1 = \rho_2 - \rho_1$ where $\rho_2$ is the permutation action of $S_4$ on the $2$-subsets of $[4]$ and $\rho_1$ is the \emph{natural representation} given by the action of $S_4$ on $[4]$.

    Let $\sigma = (1 \ 2) \in S_4$ be a transposition, which has cycle type $(2,1,1)$. Let $\rho$ denote the representation of $S_4$ by permutation matrices. The evaluation $H^*(\sigma) = 1 + t^2$ tells us that the Ehrhart series of the fixed polytope $(\Delta_{2,4})_\sigma$, see Figure~\ref{fig: hypersimplex 24 slice}, is given by
    \[
    {\rm ES}_{(\Delta_{2,4})_\sigma}[t] = 
    {\rm EES}_{\Delta_{2,4}}[t](\sigma) = 
    \frac{1 + t^2}{\det(I_4 - t\rho(\sigma))} = \frac{1 + t^2}{(1-t)^2(1 - t^2)}
    = \frac{1+2t+2t^2+2t^3+t^4}{(1-t^2)^3}.
    \]
    
    \begin{figure}[t]
        \begin{subfigure}{.45\textwidth}
        \centering
        \begin{tikzpicture}[thick,scale=2]
            \coordinate (a) at (-1, 0);
            \coordinate (b) at (0, 1);
            \coordinate (c) at (1, 0);
            \coordinate (d) at (0, -1);
            \coordinate (e) at (0.25, 0.25);
            \coordinate (f) at (-0.25, -0.25);

            \node[anchor = east] at (a) {$e_{13}$};
            \node[anchor = south] at (b) {$e_{34}$};
            \node[anchor = west] at (c) {$e_{24}$};
            \node[anchor = north] at (d) {$e_{12}$};
            \node[anchor = south west] at (e) {$e_{23}$};
            
            \begin{scope}[thick,dashed,,opacity=0.4]
                \draw (a) -- (e) -- (b);
                \draw (c) -- (e) -- (d);
            \end{scope}
            
            \draw[fill={blue!80},opacity=0.55] (a) -- (b) -- (f);
            \draw[fill={blue!70},opacity=0.55] (b) -- (c) -- (f);
            \draw[fill={blue!90},opacity=0.55] (c) -- (d) -- (f);
            \draw[fill={blue},opacity=0.5] (d) -- (a) -- (f);
            \draw (a) -- (f) -- (b);
            \draw (c) -- (f) -- (d);
            \draw (a) -- (b) -- (c) -- (d) -- cycle;

            \node[anchor = north east] at (f) {$e_{14}$};
            
        \end{tikzpicture}
        \caption{Hypersimplex $\Delta_{2,4}$}
        \label{fig: hypersimplex 24 whole}
        \end{subfigure} 
        \begin{subfigure}{.45\textwidth}
        \centering
        \begin{tikzpicture}[thick,scale=2]
            \coordinate (a) at (-1, 0);
            \coordinate (b) at (0, 1);
            \coordinate (c) at (1, 0);
            \coordinate (d) at (0, -1);
            \coordinate (e) at (0.25, 0.25);
            \coordinate (f) at (-0.25, -0.25);
            \coordinate (g) at (-0.375, 0.125);
            \coordinate (h) at (0.375, -0.125);
            
            \node[anchor = south] at (b) {$e_{34}$};
            \node[anchor = north] at (d) {$e_{12}$};
            \node[anchor = east] at (g) {$\frac 12 (e_{13} + e_{23})$};
            \node[anchor = west] at (h) {$\frac 12 (e_{24} + e_{14})$};
            
            \begin{scope}[thick,dashed,opacity=0.5]
                \draw (a) -- (e) -- (b);
                \draw (c) -- (e) -- (d);
                \draw (a) -- (f) -- (b);
                \draw (c) -- (f) -- (d);
                \draw (a) -- (b) -- (c) -- (d) -- cycle;
            \end{scope}
            
            \draw[fill={blue!80},opacity=0.5] (b) -- (h) -- (d) -- (g) -- cycle;
            \draw (b) -- (h) -- (d) -- (g) -- cycle;
        \end{tikzpicture}
        \caption{Fixed polytope $(\Delta_{2,4})_{(1 \ 2)}$}
        \label{fig: hypersimplex 24 slice}
        \end{subfigure}
        
        \caption{Fixed polytopes of the hypersimplex under the action of $S_4$ in Example~\ref{ex: hypersimplex 24}}
        \label{fig: hypersimplex 24}
    \end{figure}
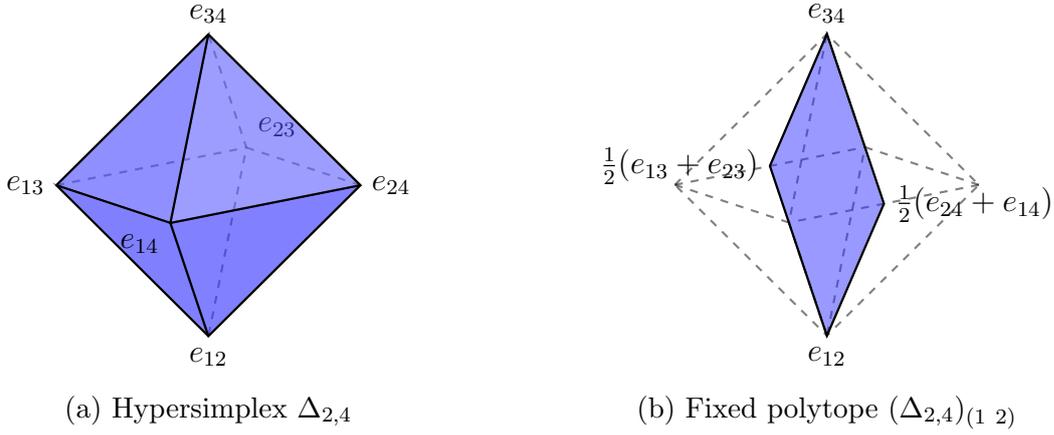
    
    \begin{table}[t]
        \centering
        \begin{tabular}{cccccc}
        \toprule
           \multirow{2}{*}{Character} & \multicolumn{5}{c}{Cycle type} \\
            & $(1,1,1,1)$ & $(2,1,1)$ & $(2,2)$ & $(3,1)$ & $(4)$ \\
        \midrule
           $H^*_0$ & $1$ & $1$ & $1$ & $1$  & $1$ \\
           $H^*_1$ & $2$ & $0$ & $2$ & $-1$ & $0$ \\
           $H^*_2$ & $1$ & $1$ & $1$ & $1$  & $1$ \\
        \midrule
           $\rho_1$ & $4$ & $2$ & $0$ & $1$ & $0$ \\
           $\rho_2$ & $6$ & $2$ & $2$ & $0$ & $0$ \\
        \bottomrule
        \end{tabular}
        \caption{Coefficients of $H^*({\Delta_{2,4}, S_4)}$ and characters in Example~\ref{ex: hypersimplex 24}}
        \label{tab: coefficients of H* 24}
    \end{table}
    
\end{example}

\subsection{Stirling numbers of the second kind}

In Section~\ref{sec: non-hyp dosps}, we count families of DOSPs. For this, we require \emph{Stirling numbers of the second kind}, which, for brevity, we refer to simply as \emph{Stirling numbers}. Given non-negative integers $n$ and $k$, the Stirling number $\stirling nk$ is the number of partitions of the set $[n]$ into $k$ non-empty unlabelled parts. For example the set $[3]$ is partitioned into $2$ non-empty parts in three ways:
$
    \{1,23\},\, \{2,13\},\, \{3,12\}. 
$
Therefore $\stirling 32 = 3$. For further background on Stirling numbers, we refer the reader to \cite{stanley2011enumerative}. 

\begin{proposition}[{\cite[Section~1.9]{stanley2011enumerative}}]\label{prop: stirling2}
    The Stirling numbers satisfy the relationship
    \[
    \sum_{k=0}^n\stirling{n}{k} (x)_k = x^n
    \]
    for all $x$, where $(x)_k=\prod_{i=0}^{k-1} (x-i)$ denotes the $k$-th falling factorial.
\end{proposition}

\section{Coefficients of the equivariant \texorpdfstring{$H^*$}{H*}-polynomial}\label{sec: coeffs of H* poly}

The main result of this section is Theorem~\ref{thm: H* coefficient in terms of representations}, which is a formula for the coefficients of the equivariant $H^*$-polynomial of the hypersimplex $\Delta_{k,n}$ under the action of $S_n$. The formula is given in terms of the following combinatorial families of objects.

\begin{definition}\label{def: Phi set and I_h set}
    Fix $\sigma \in S_n$ a permutation with cycle type $(s_1, \dots, s_r)$. For each $0 < k < n$ and $m \ge 0$, we define the set of functions
    \[
        \Phi_k(\sigma, m) = \left\{
            f \colon [r] \rightarrow \{0,1,\dots,k-1\} : \sum_{i = 1}^r f(i)s_i = m
        \right\}.
    \]
    By convention, we define $\Phi_k(\sigma, m) = \emptyset$ for all $m < 0$.
    For each $h \ge 0$, we define the set
    \[
        \mathcal I_h = \left\{
            I = (I_1,I_2, \dots, I_{k-1}) \in \ZZ^{k-1}_{\ge 0} : \sum_{i = 1}^{k-1} i \cdot I_i = h 
        \right\}.
    \]
    For each $I \in \mathcal I_h$, we write $|I| = I_1 + I_2 + \dots + I_{k-1}$.
\end{definition}

We observe that the map $\sigma \mapsto |\Phi_k(\sigma, m)|$ is a permutation character of $S_n$.

\begin{proposition}\label{prop: Phi set is perm rep}
    Fix $0 < k < n$ and $m \ge 0$. Let $\Chi$ be the permutation character of $S_n$ acting on the set of function functions $\{f \colon [n] \rightarrow \{0,1, \dots, k-1\} : \sum_{i = 1}^n f(i) = m\}$ by $(\sigma \cdot f)(i) = f(\sigma^{-1}(i))$. Then $\Chi(\sigma) = |\Phi_k(\sigma, m)|$.
\end{proposition}

\begin{proof}
    Follows immediately from the definition.
\end{proof}

The main result of this section is the following formula for the coefficients of $H^*(\Delta_{k,n}; S_n)$ in terms of the permutation characters $|\Phi_k(\sigma, m)|$ and the elements of $\mathcal I_h$.

\begin{theorem}\label{thm: H* coefficient in terms of representations}
    Fix $0 < k < n$ and $\sigma \in S_n$. For each $i \in [n]$, denote by $\lambda_i$ the number of cycles of $\sigma$ of length $i$. For each $m \ge 0$, let $H^*_m$ be the coefficient of $t^m$ in $H^*(\Delta_{k,n}; S_n)[t]$. Then
    \[
    H^*_m(\sigma) = \sum_{h = 0}^{k-1}\left(
    \sum_{I \in \mathcal I_h} (-1)^{|I|} \binom{\lambda_1}{I_1} \cdots \binom{\lambda_{k-1}}{I_{k-1}}
    \right)|\Phi_{k-h}(\sigma, m(k-h) - h)|.
    \]
\end{theorem}

We will give a proof of the theorem at the end of the section after we have introduced notation and developed the necessary tools. For now, an immediate consequence of this theorem is a combinatorial proof that the $H^*$-series $H^*(\Delta_{k,n}; S_n)$ is a polynomial and, moreover, we can write down its degree. 

\begin{corollary}\label{cor: degree of H* poly}
   Fix $0 < k < n$. We have that $H^*(\Delta_{k,n}; S_n)$ is a polynomial of degree $\lfloor (k-1)n/k \rfloor$.
\end{corollary}

\begin{proof}
Let $\sigma \in S_n$ be a permutation with cycle type $(s_1, \dots, s_r)$ and fix $h \ge 0$. For any function $f \colon [r] \rightarrow \{0,1, \dots, k-h-1\}$, we have $\sum_{i = 1}^r f(i)s_i \le (k-h-1)n$. If $m$ satisfies $(k-1)n < km$ then we have
\[
    \sum_{i = 1}^r f(i)s_i \le (k-h-1)n < \frac{(k-h-1)km}{k-1} = 
    km - \frac{khm}{k-1} = m(k - h) - \frac{hm}{k-1} < m(k-h) -h,
\]
hence the set $\Phi_k(\sigma, m(k-h)-h)$ is empty. So, by Theorem~\ref{thm: H* coefficient in terms of representations}, the coefficient $H^*_m(\sigma) = 0$ for all $m > (k-1)n/k$. In particular, the $H^*$-series is a polynomial. 

On the other hand, let $e \in S_n$ be the identity. The set $\Phi_{k}(e, mk)$ is non-empty if and only if $m$ satisfies $(k-1)n \ge k m$. Therefore, the degree of the $H^*$-polynomial is at least $\lfloor (k-1)n/k \rfloor$, which concludes the proof of the result.
\end{proof}

\subsection{Katzman's method}

In this section, we apply the method used by Katzman \cite{Katzman2005Hilbert} to obtain a formula for the coefficients of the equivariant $H^*$-series. The main result of this section is Lemma~\ref{lem: H* coefficient for Delta_l,n}, which is the first step towards Theorem~\ref{thm: H* coefficient in terms of representations}.
We introduce two pieces of useful notation. 
Given $\sigma \in S_n$ with cycle type $(s_1, \dots, s_r)$, we define the formal power series
\[
    u^\sigma = \sum_{i \ge 0} u^\sigma_i t^i := \prod_{i = 1}^k (1 + t^{s_i} + t^{2s_i} + \dots) =  \prod_{i = 1}^k \frac{1}{1 - t^{s_i}} \in \ZZ[[t]].
\]
If $\sigma$ is clear from context, then we write $u$ for $u^\sigma$ and $u_i$ for $u^\sigma_i$. For each subset $S \subseteq [r]$, we write $\Sigma^\sigma S = \sum_{i \in S} s_i$. If the permutation $\sigma$ is clear from context then we write $\Sigma S$ for $\Sigma^\sigma S$.

\begin{lemma}\label{lem: H* coefficient for Delta_l,n}
Fix $0 < k < n$ and $\sigma \in S_n$ with cycle type $(s_1, \dots, s_r)$. For each $i \in [n]$, let $\lambda_i$ be the number of length $i$ cycles of $\sigma$. For each $m \ge 0$, the coefficient of $t^m$ in $H^*(\Delta_{k,n}; S_n)[t]$ is
\[
H^*_m(\sigma) = 
    \sum_{S \subseteq [r]} (-1)^{|S|}
    \sum_{h = 0}^{k - 1} \left(
        \sum_{I \in \mathcal I_h} (-1)^{|I|} \binom{\lambda_1}{I_1} \cdots \binom{\lambda_{k-1}}{I_{k-1}}
    \right) u_{(m-\Sigma S)(k-h) - h}.
\]
\end{lemma}

\begin{proof}
By conjugating $\sigma$, we may assume without loss of generality that 
\[
\sigma = (1 \ 2 \ \dots \ s_1)(s_1+1 \ s_1+2 \ \dots \ s_1+s_2 ) \dots (n-s_r+1 \ n-s_r+2 \ \dots \ n).
\]
Fix $d \ge 0$. We have that $(d\Delta_{k,n})_\sigma \cap \ZZ^n$ is equal to
\[
    \left\{
    (
    \underbrace{x_1, \dots, x_1}_{s_1}, 
    \underbrace{x_2, \dots, x_2}_{s_2}, \dots, 
    \underbrace{x_r, \dots, x_r}_{s_r}) \in \ZZ^n : 
    \sum_{i = 1}^r x_i s_i = k d,\, 
    0 \le x_i \le d \text{ for all } i \in [r]
    \right\}.
\] 
So there is a bijection between the set solutions $(x_1, x_2, \dots, x_r) \in \{0,1, \dots, d \}^r$ to the equation $\sum_{i = 1}^r x_i s_i = kd$ and lattice points $(d \Delta_{k,n})_\sigma \cap \ZZ^n$. Consider the polynomial
\[
    f_d(t) = \prod_{i = 1}^r (1 + t^{s_i} + t^{2s_i} + \dots + t^{ds_i}) = 
    \prod_{i = 1}^r \frac{1 - t^{(d+1)s_i}}{1 - t^{s_i}}.
\]
For each solution $(x_1, \dots, x_r)$ to the above equation, we have a term $t^{kd} = t^{x_1 s_1} t^{x_2 s_2} \dots t^{x_r s_r}$ in the expansion of $f_d(t)$. Moreover, each term $t^{kd}$ in the expansion of $f_d(t)$ arises from such a solution. So we have that $|(d \Delta_{k,n})_\sigma \cap \ZZ^n|$ is equal to the coefficient of $t^{k d}$ in $f_d(t)$, which we denote by $[f_d]_{k d}$.

For each $s_i \ge k$, we have that $t^{(d+1)s_i}$ does not divide $t^{k d}$. It follows that $[f_d]_{k d}$ is equal to the coefficient of $t^{k d}$ in the formal power series:
\begin{align*}
[f_d]_{k d} 
    &= \left[
        \prod_{j = 1}^{k - 1} (1 - t^{(d+1)j})^{\lambda_j}
        \prod_{i = 1}^r \frac{1}{1 - t^{s_i}}
        \right]_{k d} \\
    &= \left[
        \prod_{j = 1}^{k - 1} 
        \sum_{h = 0}^{\lambda_j} (-1)^h \binom{\lambda_j}{h} t^{(d+1)jh}
        \cdot u
        \right]_{k d}\\
    &= \left[
        \sum_{h = 0}^{k - 1}
            \sum_{I \in \mathcal I_h} (-1)^{|I|} \binom{\lambda_1}{I_1} \binom{\lambda_2}{I_2} \cdots \binom{\lambda_{\ell-1}}{I_{k-1}}
        t^{(d+1)h}
        \cdot u
        \right]_{k d} 
        \\
    &=
        \sum_{h = 0}^{k - 1}
            \sum_{I \in \mathcal I_h} (-1)^{|I|} \binom{\lambda_1}{I_1} \binom{\lambda_2}{I_2} \cdots \binom{\lambda_{\ell-1}}{I_{k-1}}
        u_{kd - (d+1)h}.\\
\end{align*}
So the Ehrhart series of $(\Delta_{k,n})_\sigma$ is given by
\[
    \sum_{d \ge 0} [f_d]_{kd} t^d
    = \sum_{d \ge 0} \left(
    \sum_{h = 0}^{k - 1}
            \sum_{I \in \mathcal I_h} (-1)^{|I|} \binom{\lambda_1}{I_1} \cdots \binom{\lambda_{k-1}}{I_{k-1}}
        u_{k d - h(d+1)}
    \right) t^d
    = \frac{H^*[t](\sigma)}{\prod_{i = 1}^r(1 - t^{s_i})},
\]
where the right-most equality follows from definition of the equivariant $H^*$-series and the denominator comes from $\det(I_n - t\rho(\sigma))$ where $\rho(\sigma)$ is the permutation matrix of $\sigma$ and $I_n$ is the identity matrix of size $n$. So, by clearing the denominator, we obtain a formula for the coefficients of equivariant $H^*$-series. For each $m \ge 0$, we have

\begin{align*}
H^*_m(\sigma) &= 
    \left[
        \prod_{i = 1}^r (1 - t^{s_i})
        \sum_{d \ge 0} \left(
        \sum_{h = 0}^{k - 1} 
                \sum_{I \in \mathcal I_h} (-1)^{|I|} \binom{\lambda_1}{I_1} \cdots \binom{\lambda_{k-1}}{I_{k-1}}
            u_{kd - h(d+1)}
        \right) t^d
    \right]_m \\
&= 
    \sum_{S \subseteq [r]} (-1)^{|S|}
    \sum_{h = 0}^{k - 1} \left(
        \sum_{I \in \mathcal I_h} (-1)^{|I|} \binom{\lambda_1}{I_1} \cdots \binom{\lambda_{k-1}}{i_{\ell-1}}
    \right) u_{(m-\Sigma S)(k-h) - h}.
\end{align*}
This concludes the proof of the result.
\end{proof}

We illustrate some of the main steps of the above proof with the following example.

\begin{example}
    Let $k = 3$. In this case, we consider the sets $\mathcal I_0 =  \{(0,0)\}$, $\mathcal I_1 = \{(1,0)\}$, and $\mathcal I_2 = \{(2,0), (0,1)\}$.
    So, we have
    \[
        [f_d]_{3d} = 
        \left[
        \left(
            1 - 
            \lambda_1 t^{d+1} +
            \left(
                \binom{\lambda_1}{2} - 
                \lambda_2
            \right) t^{2(d+1)}
        \right)u^\sigma
        \right]_{3 d}
        = 
        u_{3d} - 
        \lambda_1 u_{2d-1} +
        \left(
            \binom{\lambda_1}{2} - 
            \lambda_2
        \right) u_{d-2}.
    \]
    The Ehrhart series of $(\Delta_{3,n})_\sigma$ is given by
    \[
        \frac{H^*[t](\sigma)}{\prod_{i = 1}^r (1-t^{s_i})}
        =
        \sum_{d \ge 0} 
        \left(
            u_{3d} - 
            \lambda_1 u_{2d - 1} + 
            \left(
                \binom{\lambda_1}{2} - 
                \lambda_2
            \right) u_{d-2}
        \right) t^d.
    \]
    The coefficient of $t^m$ in the $H^*$-series is given by
    \[
        H^*_m(\sigma) = \sum_{S \subseteq [r]} (-1)^{|S|} \left(
            u_{3(m - \Sigma S)} - 
            \lambda_1 u_{2(m - \Sigma S) - 1} +
            \left(
                \binom{\lambda_1}{2} - 
                \lambda_2
            \right) u_{m - \Sigma S - 2}
        \right).
    \]
\end{example}

\subsection{Permutation-representation interpretation}
In this section we prove Theorem~\ref{thm: H* coefficient in terms of representations}. We recall the sets of functions $\Phi_{k}(\sigma, m)$ from Definition~\ref{def: Phi set and I_h set}. We use these sets of functions to give an interpretation of terms appearing in formula for $H^*_m$ in Lemma~\ref{lem: H* coefficient for Delta_l,n}. 

\begin{lemma}\label{lem: Delta_l,n sum in terms of reps}
    Fix $0 < k < n$ and $\sigma \in S_n$ with cycle type $(s_1, \dots, s_r)$. With our usual notation, we have
    \[
        \sum_{S \subseteq [r]} (-1)^{|S|} u_{m - k \Sigma S} = |\Phi_k(\sigma, m)|.
    \]
\end{lemma}

\begin{proof}
    We prove the result by induction on $r$. For the base case, assume $r = 1$, i.e., $\sigma$ is an $n$-cycle. We have $\Sigma \emptyset = 0$ and $\Sigma \{1\} = n$, and  
    $
        u = 1 + t^n + t^{2n} + \dots = 1/(1-t^n).
    $
    Therefore the left-hand sum is given by
    \[
    \sum_{S \subseteq [r]} (-1)^{|S|} u_{m - k \Sigma S} = 
    u_m - u_{m - k n} = \begin{cases}
        1 & \text{if } m \in \{0, n, 2n, \dots, (k-1)n \}, \\
        0 & \text{otherwise}.
    \end{cases}
    \]
    On the other hand, there are exactly $k$ functions $f \colon [r] \rightarrow \{0,1,\dots, k-1 \}$ and any such function $f$ satisfies $\sum_{i} f(i)s_i = f(1)n$. Therefore 
    \[
    |\Phi_k(\sigma, m)| = \begin{cases}
        1 & \text{if } m \in \{0, n, 2n, \dots, (k-1)n \} \\
        0 & \text{otherwise}
    \end{cases}
    = \sum_{S \subseteq [r]} (-1)^{|S|} u_{m - k \Sigma S},
    \]
    and we are done with the base case.

    For the induction step, let $\sigma$ be a permutation with cycle type $(s_1, s_2, \dots, s_{r+1})$ and assume that the result holds for any permutation with $r$ disjoint cycles. Let $\tau$ be a permutation with cycle type $(s_1, \dots, s_r)$. For ease of notation, we define $s := s_{r+1}$. 
    We note that $u^\sigma = u^\tau (1 + t^s + t^{2s} + \dots)$, so it follows that $u^\sigma_i = \sum_{j \ge 0} u^\tau_{i - sj}$.
    Then have the following chain of equalities:
    \begin{align*}
        \sum_{S \subseteq [r+1]} (-1)^{|S|} u^\sigma_{m - k \Sigma^\sigma S}
            &= \sum_{S \subseteq [r]} (-1)^{|S|} 
                \left(
                    u^\sigma_{m - k \Sigma^\sigma S} - u^\sigma_{m - k \Sigma^\sigma S - k s}
                \right)\\
            &= \sum_{S \subseteq [r]} (-1)^{|S|} 
                \left(
                    u^\sigma_{m - k \Sigma^\tau S} - u^\sigma_{m - k \Sigma^\tau S - k s}
                \right) \\ 
            &= \sum_{S \subseteq [r]} (-1)^{|S|} \sum_{j \ge 0}
                \left(
                    u^\tau_{m - k \Sigma^\tau S - sj} - u^\tau_{m - k \Sigma^\tau S - s(j + k)}
                \right) \\ 
            &= \sum_{S \subseteq [k]} (-1)^{|S|} \sum_{j = 0}^{k-1}
                    u^\tau_{m - k \Sigma^\tau S - sj}
            = \sum_{j = 0}^{k-1} |\Phi_k(\tau, m - sj)|.
    \end{align*}
    To conclude the proof, we note that there is a natural bijection between the sets
    \begin{align*}
        \Phi_k(\sigma, m) &\leftrightarrow \bigsqcup_{j = 0}^{k-1} \Phi_k(\tau, m-sj), \\
        f &\mapsto f|_{[r]} \in \Phi_k(\tau, m - sf(k)), \\
        \left(i \mapsto \begin{cases}
            f(i) & \text{if } i \in [r], \\
            j & \text{if } i = r+1
        \end{cases}
        \right)
        &\mapsfrom f \in \Phi_k(\tau, m - sj) \text{ for some } j \in \{0,1,\dots,k-1 \}.
    \end{align*}
    It follows that $|\Phi_k(\sigma, m)| = \sum_{j = 0}^{k-1} |\Phi_k(\tau, m-sj)| = \sum_{S \subseteq [r+1]} (-1)^{|S|} u^\sigma_{m - k \Sigma^\sigma S}$. This concludes the proof of the result.
\end{proof}

With this result, we give a proof of Theorem~\ref{thm: H* coefficient in terms of representations}.

\begin{proof}[Proof of Theorem~\ref{thm: H* coefficient in terms of representations}]
By Lemma~\ref{lem: Delta_l,n sum in terms of reps}, we have
\[
    \sum_{S \subseteq [r]} (-1)^{|S|} u_{m(k-h)-h -(k-h)\Sigma S} = |\Phi_{k-h}(\sigma, m(k-h)-h)|.
\]
By Lemma~\ref{lem: H* coefficient for Delta_l,n}, we have
\begin{align*}
H^*_m(\sigma) &= 
    \sum_{S \subseteq [r]} (-1)^{|S|}
    \sum_{h = 0}^{k - 1} \left(
        \sum_{I \in \mathcal I_h} (-1)^{|I|} \binom{\lambda_1}{I_1} \cdots \binom{\lambda_{k-1}}{I_{k-1}}
    \right) u_{(m-\Sigma S)(k-h) - h}\\
    &=
    \sum_{h = 0}^{k - 1} \left(
        \sum_{I \in \mathcal I_h} (-1)^{|I|} \binom{\lambda_1}{I_1} \cdots \binom{\lambda_{k-1}}{I_{k-1}}
    \right) 
    \sum_{S \subseteq [r]} (-1)^{|S|} u_{m(k-h)-h -(k-h)\Sigma S}\\
    &=
    \sum_{h = 0}^{k - 1} \left(
        \sum_{I \in \mathcal I_h} (-1)^{|I|} \binom{\lambda_1}{I_1} \cdots \binom{\lambda_{k-1}}{I_{k-1}}
    \right) |\Phi_{k-h}(\sigma, m(k-h)-h)|.
\end{align*}
This concludes the proof.
\end{proof}

\section{Decorated ordered set partitions}\label{sec: DOSP thm}

In this section, we show that $H^*(\Delta_{k,n}; S_n)[1]$ is the permutation character of $S_n$ acting naturally on the set of hypersimplicial $(k,n)$-DOSPs. From Definition~\ref{def: Phi set and I_h set}, we recall the definition of the set $\mathcal I_h$. Our main result gives a formula for the number of hypersimplicial $\sigma$-fixed $(k,n)$-DOSPs.

\begin{theorem}\label{thm: H*[1] is perm character of DOSPs under Sn}
    Let $2 \le k < n$ and denote $H^* = H^*(\Delta_{k,n}; S_n)$. Then $H^*[1]$ is equal to the permutation character of the action of $S_n$ on hypersimplicial $(k,n)$-DOSPs. Let $\sigma \in S_n$ be a permutation with $r$ disjoint cycles, and, for each $i \in [n]$, write $\lambda_i$ for the number of length $i$ cycles of $\sigma$. Then the number of $\sigma$-fixed hypersimplicial $(k,n)$-DOSPs is
    \[
    H^*[1](\sigma) = g
    \sum_{h = 0}^{k-1}
    \left( 
        \sum_{I \in \mathcal I_{h}}
        (-1)^{|I|}
        \binom{\lambda_1}{I_1} \cdots \binom{\lambda_{k-1}}{I_{k-1}}
    \right)
    (k - h)^{r - 1}
    \]
    where $g = \gcd(\{k\} \cup \{i \in [n] : \lambda_i \ge 1\})$.
\end{theorem}

Let us outline the results of the next two sections, which we use to prove Theorem~\ref{thm: H*[1] is perm character of DOSPs under Sn}. In Section~\ref{sec: all DOSPs}, we use Theorem~\ref{thm: H* coefficient in terms of representations} to show that the above formula for $H^*[1]$ holds (Corollary~\ref{cor: H*[1] coeff short}). We show that the number of $\sigma$-fixed $(k,n)$-DOSPs, including non-hypersimplicial DOSPs, is equal to $gk^{r-1}$ (Corollary~\ref{cor: size of Phi}) and observe that this is equal to the $h=0$ term in the above sum. In Section~\ref{sec: non-hyp dosps}, we give an explicit formula for the number of $\sigma$-fixed non-hypersimplicial DOSPs in Lemma~\ref{lem: count non-hyp DOSPs}. The proof of Theorem~\ref{thm: H*[1] is perm character of DOSPs under Sn} then follows by simplifying the formula using a modified version of the falling factorial identity for Stirling numbers. In the remainder of this section, we give an alternative but equivalent definition for DOSPs under the action of $S_n$ and define a notion of \emph{directed distance} within a DOSP.

\medskip

\noindent \textbf{Alternative DOSP definition.}
Fix $k < n$. Let $\Psi = \set{f\colon [n]\to\ZZ/k\ZZ}$ be the set of functions modulo the equivalence relation $f\sim g$ if and only if $f-g$ is constant.
There is an action of $S_n$ on $\Psi$ given by $(\sigma\cdot f)(i) = f(\sigma^{-1}(i))$ for each $\sigma\in S_n$ and $f\in\Psi$.
We now describe the natural $S_n$-set isomorphism between $\Psi$ and $(k,n)$-DOSPs.
Given a DOSP $D=((L_1,\ell_1),(L_2,\ell_2),\ldots,(L_t,\ell_t))$, its corresponding function is $f_D$ such that $f_D(i)=0$ if $i \in L_1$ and $f_D(i)=\ell_1+\ell_2+\cdots+\ell_{j-1}$ if $i\in L_j$ with $j \ge 2$.
It is straightforward to check that the map $D\mapsto f_D$ is an isomorphism of $S_n$-sets.

\begin{definition}\label{def: dist in dosp and winding no}
    Let $D = ((L_1, \ell_1), \dots, (L_r, \ell_r))$ be a $(k,n)$-DOSP. Let $i, j \in [n]$. We define the \emph{directed distance} $d_D(i,j)$ from $i$ to $j$ in $D$ as follows. Without loss of generality, we may assume $i \in L_1$. Suppose that $j \in L_u$ for some $u \in [r]$. Then $d_D(i,j) := \ell_1 + \ell_2 + \dots + \ell_{u-1} \in \{0,1, \dots, k-1\}$. The \emph{winding number} of $D$ is $w(D) =  (d_D(1,2) + d_D(2,3) + \dots + d_D(n-1, n) + d_D(n,1)) / k$. If the DOSP is given as a function $f: [n] \rightarrow \ZZ/k\ZZ$, then, for each $i,j \in [n]$, the directed distance is $d_f(i,j) = f(j) - f(i)$ where we take the representative of $f(j) - f(i)$ in $\{0,1, \dots, k-1\}$.

    Fix a permutation $\sigma\in S_n$.
    Given a $\sigma$-fixed DOSP $f \colon [n] \rightarrow \ZZ/k\ZZ$, we define the \emph{turning number} $\tau\in\ZZ/k\ZZ$ of $f$ with respect to $\sigma$ such that $\tau + f(i) = (\sigma\cdot f)(i)$ for any $i\in [n]$.
    This notion is well-defined, i.e., $\tau$ does not depend on $i$, since $\sigma$ fixes $f$.
\end{definition}

\begin{example}\label{example: turning number}
    Let $k = 6$, $n = 10$, and $\sigma = (1 \ 2 \ 3 \ 4 \ 5 \ 6)(7 \ 8 \ 9 \ 10) \in S_{10}$. Let us compute the turning number of the DOSP $D = ((1\,3\,5, 1), (7\,9, 2), (2\,4\,6, 1), (8\,10, 2))$ under $\sigma$, which is depicted in Figure~\ref{fig: turning number}. If we apply $\sigma$ to $D$ then we obtain 
    $\sigma(D) = ((2\,4\,6, 1), (8\,10, 2), (1\,3\,5, 1), (7\,9, 2))$, which is equivalent to $D$. The equivalence of $\sigma(D)$ and $D$ can be observed in the figure: the right DOSP is obtained by turning the left DOSP clockwise by $3$ spaces. In general, the turning number is the number of spaces we turn $D$ to obtain $\sigma(D)$. So, the turning number of $D$ with respect to $\sigma$ is $3$.

    \begin{figure}[t]
    \centering
    \begin{tikzpicture}[scale = 0.9]
    
    \draw[thick] (6,12.25) circle (2.75);
    
    \vertex [label={$\set{1,3,5}$}] at ({6+2.75*cos(90)},{12.25+2.75*sin(90)}) {};
    \vertex [label={[anchor=south west,] $\set{7,9}$}] at ({6+2.75*cos(30)},{12.25+2.75*sin(30)}) {};
    \smallvertex [] at ({6+2.75*cos(-30)},{12.25+2.75*sin(-30)}) {};
    \vertex [label={[yshift=-27]$\set{2,4,6}$}] at ({6+2.75*cos(-90)},{12.25+2.75*sin(-90)}) {};
    \vertex [label={[xshift=-15, yshift=-25]$\set{8,10}$}]
    at ({6+2.75*cos(-150)},{12.25+2.75*sin(-150)}) {};
    \smallvertex []
    at ({6+2.75*cos(-210)},{12.25+2.75*sin(-210)}) {};
    
    \draw[-{Latex[length=4mm, width=2.5mm]}] ({6+2.2*cos(90)},{12.25+2.2*sin(90)}) arc[start angle=90, end angle=-210, radius=2.2];

    \draw[thick] (16,12.25) circle (2.75);
    
    \vertex [label={$\set{2,4,6}$}] at ({16+2.75*cos(90)},{12.25+2.75*sin(90)}) {};
    \vertex [label={[anchor=south west,] $\set{8,10}$}] at ({16+2.75*cos(30)},{12.25+2.75*sin(30)}) {};
    \smallvertex [] at ({16+2.75*cos(-30)},{12.25+2.75*sin(-30)}) {};
    \vertex [label={[yshift=-27]$\set{1,3,5}$}] at ({16+2.75*cos(-90)},{12.25+2.75*sin(-90)}) {};
    \vertex [label={[xshift=-15, yshift=-25]$\set{7,9}$}]
    at ({16+2.75*cos(-150)},{12.25+2.75*sin(-150)}) {};
    \smallvertex []
    at ({16+2.75*cos(-210)},{12.25+2.75*sin(-210)}) {};
    
    \draw[-{Latex[length=4mm, width=2.5mm]}] ({16+2.2*cos(90)},{12.25+2.2*sin(90)}) arc[start angle=90, end angle=-210, radius=2.2];
    
    \draw [-{Latex[length=4mm, width=2.5mm]}] (10,12.25) -- (12,12.25);
    \node [font=\Large] at (11,12.75) {$\sigma$};
    \end{tikzpicture}
    
    \caption{The DOSPs in Example~\ref{example: turning number}. The DOSP $\sigma(D)$ (right) is obtained from $D$ (left) by turning it $3$ spaces clockwise, hence the turning number of $D$ with respect to $\sigma$ is $3$.}
    \label{fig: turning number}
\end{figure}
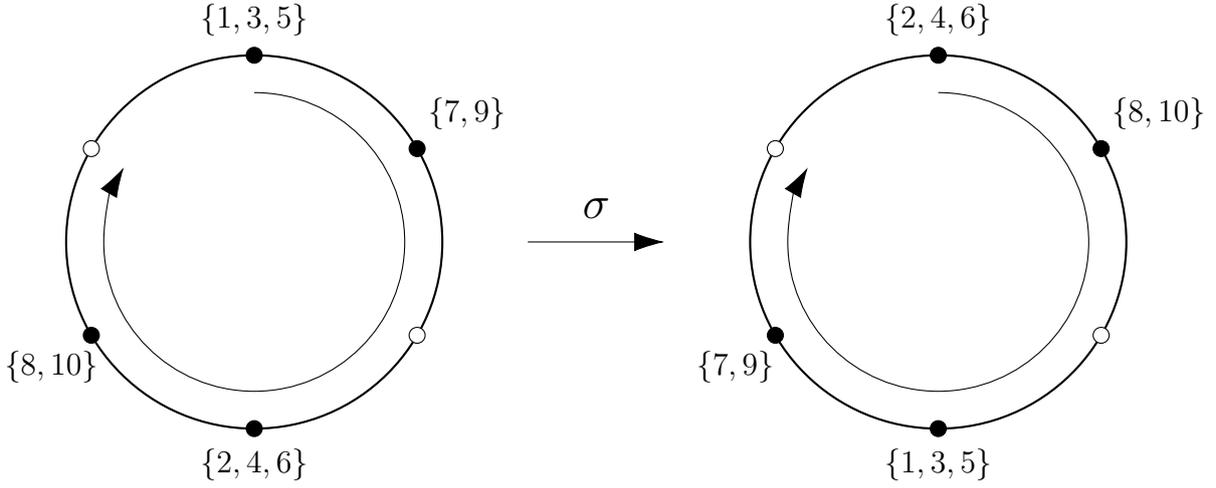
    
\end{example}

\begin{lemma}\label{lem: turning number and g}
    Fix $2\leq k < n$ and let $\sigma\in S_n$ be a permutation with cycle type $(s_1,s_2,\dots,s_r)$.
    Let $g=\gcd(s_1,\ldots,s_r,k)$.
    If $D$ is a $\sigma$-fixed $(k,n)$-DOSP with turning number $\tau$, then $g \tau=0$.
\end{lemma}

\begin{proof}
    Suppose that $f \colon [n] \rightarrow \ZZ/k\ZZ$ is a $\sigma$-fixed DOSP with non-zero turning number $\tau$.
    Notice that $(\sigma^{s_i} \cdot f)(q_i)=f(q_i)$,
    so we have $s_i \tau = 0$ for every $1\leq i\leq r$.
    Since $\tau \in \ZZ/k\ZZ$ we have $k \tau=0$, and it follows that $g \tau = 0$.
\end{proof}

\subsection{Interpreting terms with DOSPs}\label{sec: all DOSPs}

Throughout, we fix $k < n$ and write $H^*[t] = H^*(\Delta_{k,n}; S_n)[t]$ for the equivariant $H^*$-polynomial.
By Corollary~\ref{cor: degree of H* poly}, let $d = \lfloor(k - 1)n/k \rfloor$ be the degree of $H^*$. By Theorem~\ref{thm: H* coefficient in terms of representations}, we have that
\[
H^*[1](\sigma) 
    = \sum_{m = 0}^{d}\sum_{h = 0}^{k - 1} \left(
        \sum_{I \in \mathcal I_h} (-1)^{|I|} \binom{\lambda_1}{I_1} \cdots \binom{\lambda_{k-1}}{I_{k-1}}
    \right) |\Phi_{k-h}(\sigma, m(k-h) - h)|. \\
\]
The $h = 0$ term in the above sum is
$
\sum_{m = 0}^{d} |\Phi_\ell(\sigma, m \ell)|,
$
which we will show corresponds to the number of $\sigma$-fixed $(k,n)$-DOSPs. In subsequent sections, we show that the remaining terms count the non-hypersimplicial $(k,n)$-DOSPs. Hence, the overall value $H^*[1]$ is the number of $\sigma$-fixed hypersimplicial $(k,n)$ DOSPs.

For ease of notation, let us write functions as tuples. The function $f \colon [n] \rightarrow \{0,\dots, k-1\}$ is written as $(f(1), f(2), \dots, f(n)) \in \{0, \dots, k-1\}^n$.

\begin{example}
    Consider the case $n = 6$, $k = 3$, and take the permutation $\sigma = (1\ 2\ 3\ 4)(5\ 6)$. In this case we have the functions
    \[
    \bigcup_{m = 0}^4 \Phi_3(\sigma, 3m) = \{(0,0,0,0,0,0), (1,1,1,1,1,1), (2,2,2,2,2,2) \}.
    \]
    Indeed there are three $\sigma$-fixed DOSPs, which are given by
    \[
    D_1 = ((123456,3)), \quad
    D_2 = ((1234,2), (56,1)), \quad
    D_3 = ((1234,1), (56,2)).
    \]
\end{example}

\begin{lemma}\label{lem: count sigma fixed dosps}
    Fix $0 < k < n$. Let $\sigma \in S_n$ be a permutation with cycle type $(s_1, \dots, s_r)$ and define $g = \gcd(k, s_1, s_2, \dots, s_r)$. Then the number of $\sigma$-fixed DOSPs is $g k^{r-1}$. In particular, there is a bijection between the set of $\sigma$-fixed DOSPs and the set
    \[
        \{(\alpha_1, \alpha_2, \dots, \alpha_r) : 0 \le \alpha_1 \le g-1,\, 0 \le \alpha_i \le k - 1,\, 2 \le i \le r\}.
    \]
\end{lemma}

\begin{proof}
    Throughout, we consider DOSPs as functions $f \colon [n] \rightarrow \ZZ/k\ZZ$ up to equivalence. If $\sigma$ is the identity permutation, then every DOSP is fixed by $\sigma$. The number of $\sigma$-fixed DOSPs is $k^{n-1}$, because we may take $f(1) = 0$ and freely choose the values $f(i) \in \ZZ/k\ZZ$ for each $i \in \{2,3, \dots, n\}$. Note that each such choice gives a distinct DOSP.

    Now suppose that $\sigma$ is not the identity. Let $C_1, C_2, \dots, C_r$ be the cycle sets of $\sigma$. Without loss of generality, we may assume that $s_1 > 1$ and $1 \in C_1$. We define $q_1 = \sigma(1)$ and, for each $i \in \{2, \dots, r\}$, let us fix a distinguished element $q_i \in C_i$.
    Let $f \colon [n] \rightarrow \ZZ/k\ZZ$ be a $\sigma$-fixed $(k,n)$-DOSP. Without loss of generality we assume that $f(1) = 0$. We will show that the sequence of integers
    \[
        \alpha = (\alpha_1, \alpha_2, \dots, \alpha_k)= (d_f(1,q_1),\, d_f(1,q_2),\, \dots,\, d_f(1,q_r))
    \]
    uniquely determines the DOSP.
    By assumption $f(1) = 0$. By definition, we have $f(q_1) = d_f(1,q_1)$. Since $D$ is invariant under $\sigma$, it follows that 
    \[
        d_D(1,\sigma(1)) = d_D(\sigma(1),\sigma^2(1)) = \dots = d_D(\sigma^{s_1-1}(1), \sigma^{s_1}(1)) = d_D(1,q_1).
    \]
    So the value $f(\sigma^i(1))$ for each element of $C_1 = \{1, \sigma(1), \sigma^2(1), \dots, \sigma^{s_1 - 1}(1) \}$ is determined by $d_f(1,q_1)$. Explicitly, we have $f(\sigma^i(1)) = i  \cdot d_f(1,q_1) \mod{k}$ for all $i \ge 0$.
    By a similar argument, the value $f(\sigma^i(q_2))$ for each element of $C_2 = \{q_2, \sigma(q_2), \sigma^2(q_2), \dots, \sigma^{s_2 - 1}(q_2)\}$ is determined by $d_f(1, q_2)$. To see this, observe that $f(q_2) = d_f(1, q_2)$ and, since $f$ is invariant under $\sigma$, it follows that $d_f(q_2, \sigma(q_2)) = d_f(1,\sigma(1))$. We deduce that the DOSP $f$ is uniquely determined by $\alpha$.

    We now consider the possible vectors $\alpha$. By definition, we have that $d_f(1,q_1) = d_f(1, \sigma(1))$ is the turning number of $f$. By Lemma~\ref{lem: turning number and g} we have that $g \cdot d_f(1, q_1) \equiv 0 \mod k$. The possible values for $d_f(1,q_1) \cdot g$ are 
    $
    \beta \cdot k
    $ for each $\beta \in \{0,1, \dots, g - 1\}$. Hence, the possible values for $d_f(1,q_1)$ are $\beta k / g$ for each $0 \le \beta < g$.
    For each such value of $\alpha_1 = d_f(1,q_1)$, we may freely choose the values $\alpha_2, \dots, \alpha_k$ in $\{0,1,\dots, k-1\}$. Each choice gives a distinct DOSP and every $\sigma$-fixed $(k,n)$-DOSP arises in this way. So the number of DOSPs is $gk^{r-1}$.
\end{proof}

\begin{proposition}\label{prop: count size of Phi}
    Let $\sigma \in S_n$ be a permutation with cycle type $(s_1, s_2, \dots, s_r)$, and fix $k \in [n]$. Define 
    \[
        \Phi = \left\{(f_1, f_2, \dots, f_r) \in \{0,1, \dots, k-1\}^r : \sum_{i = 1}^r f_i s_i \equiv 0 \mod k \right\},
    \]
    and let $g := \gcd(s_1, s_2, \dots, s_r, k)$.
    Then $|\Phi| = g k^{r-1}$.
\end{proposition}

\begin{proof}
    Consider the homomorphism of abelian groups
    \[
        \varphi \colon (\ZZ/k\ZZ)^r
        \rightarrow \ZZ/k\ZZ, \quad (f_1, f_2, \dots, f_r) \mapsto \sum_{i = 1}^r f_i s_i.
    \]
    The image of $\varphi$ is the subgroup of $\ZZ/k\ZZ$ generated by $s_1, s_2, \dots, s_r$. By Bezout's identity
    \[
        \langle s_1, s_2, \dots, s_k \rangle = \langle \gcd(s_1, s_2, \dots, s_r, k) \rangle = \langle g \rangle \subseteq \ZZ/k\ZZ.
    \]
    So we have $|\im(\varphi)| = k / g$. Therefore
    $
        |\Phi| = |\ker(\varphi)| = \frac{k^r}{k / g} = g k^{r-1}
    $ and we are done.
\end{proof}

The two results above, give us the following.

\begin{corollary}\label{cor: size of Phi}
Fix $0 < k < n$. Let $\sigma \in S_n$ be a permutation with cycle type $s_1, \dots, s_r$ and define $g = \gcd(k, s_1, s_2, \dots, s_r)$. Recall the set $\Phi$ from the statement of Proposition~\ref{prop: count size of Phi}.
There is a bijection between the set of all $\sigma$-fixed DOSPs and the set
\[
    \Phi = \bigcup_{m = 0}^{\lfloor (k-1)n/k \rfloor } \Phi_k(\sigma,  mk). \quad 
    \text{ Therefore}
    \quad
    \sum_{m \ge 0} |\Phi_k(\sigma, mk)| = g k^{r-1}.
\]
\end{corollary}

\begin{proof}
    Suppose that $\sigma$ has cycle type $(s_1, s_2, \dots, s_r)$.
    By Lemma~\ref{lem: count sigma fixed dosps}, the number of $\sigma$-fixed $(k,n)$-DOSPs is $g k^{r-1}$. By Proposition~\ref{prop: count size of Phi}, we have that $|\Phi| = g k^{r-1}$, and we are done.
\end{proof}

\begin{lemma}\label{lem: Phi count with g}
    Let $\sigma \in S_n$ be a permutation with cycle type $(s_1, s_2, \dots, s_r)$. Fix $h \in \{0,1,\dots, k-1\}$. Define $g' = \gcd(s_1, s_2, \dots, s_r, k-h)$. Then we have
    \[
        \sum_{m \ge 0} |\Phi_{k-h}(\sigma, m(k-h)-h)| = \begin{cases}
            g' (k-h)^{r-1} & \text{if } g' \text{ divides } h, \\
            0 & \text{otherwise.}
        \end{cases}
    \]
\end{lemma}

\begin{proof}
    Consider the homomorphism of abelian groups 
    \[
        \varphi \colon (\ZZ/(k-h)\ZZ)^r \rightarrow \ZZ/(k-h)\ZZ, \quad (f_1, f_2, \dots, f_r) \mapsto \sum_{i = 1}^r f_i s_i.
    \]
    Observe that there is a natural bijection between $\bigcup_{m \ge 0} \Phi_{k-h}(\sigma, m(k-h)-h)$ and the set $\Phi' := \{f \in (\ZZ/(k-h)\ZZ)^r : \varphi(f) = -h \mod (k-h)\}$. By Bezout's identity, it follows that the image of $\varphi$ is principally generated by $g'$.
    It follows that
    \[
    |\Phi'| = \begin{cases}
        |\ker(\varphi)| = g'(k-h)^{r-1} & \text{if } -h \in \langle g' \rangle \subseteq \ZZ/(k-h)\ZZ, \\
        0 & \text{otherwise.}
    \end{cases}
    \]
    In the first case, we have that $-h \in \langle g' \rangle $ if and only if $g'$ divides $h$.
\end{proof}

\begin{proposition}\label{prop: gh divides h iff g0 divides h}
    Fix $0 \le h < k$ and some positive integers  $s_1, \dots, s_r$. For each $0 \le i < k$ define $g_i = \gcd(k-i, s_1, \dots, s_r)$. Then $g_h | h$ if and only if $g_0 | h$.
\end{proposition}

\begin{proof}
    Define $\tilde g = \gcd(s_1, \dots, s_r)$, so $g_h = \gcd(k-h, \tilde g)$ and $g_0 = \gcd(k, \tilde g)$. We have $g_h | h$ if and only if $(k - h) | h$ and $\tilde g | h$. On the other hand $g_0 | h$ if and only if $k | h$ and $\tilde g | h$. So it suffices to show that $(k - h) | h$ if and only if $k | h$, which easily follows from the assumption that $k > h \ge 0$.
\end{proof}

\begin{proposition}\label{prop: H^*[1] coeff}
    Let $\sigma \in S_n$ be a permutation with cycle type $(s_1, \dots, s_r)$.
    For each $h \ge 0$, define $g := \gcd(k, s_1, \dots, s_r)$ and $g_h := \gcd(k-h, s_1, \dots, s_r)$.
    We have
    \[
    H^*[1](\sigma) = \sum_{h = 0}^{k-1} \left(
    \sum_{I \in \mathcal I_h} (-1)^{|I|} \binom{\lambda_1}{I_1} \cdots \binom{\lambda_{k-1}}{I_{k-1}}
    \right) g_h (k - h)^{r-1} d(g, h),
    \]
    where
    $d(g, h) = 1$ if $g$ divides $h$ and $d(g, h) = 0$ otherwise.
\end{proposition}

\begin{proof}
    Follows immediately from Theorem~\ref{thm: H* coefficient in terms of representations}, Lemma~\ref{lem: Phi count with g}, and Proposition~\ref{prop: gh divides h iff g0 divides h}.
\end{proof}

\begin{corollary}\label{cor: H*[1] coeff short}
    Let $\sigma \in S_n$ be a permutation with cycle type $(s_1, \dots, s_r)$.
    For each $h \ge 0$, define $g := \gcd(k, s_1, \dots, s_r)$.
    We have
    \[
    H^*[1](\sigma) = g\sum_{h = 0}^{k-1} \left(
    \sum_{I \in \mathcal I_h} (-1)^{|I|} \binom{\lambda_1}{I_1} \cdots \binom{\lambda_{k-1}}{I_{k-1}}
    \right) (k - h)^{r-1}.
    \]
\end{corollary}

\begin{proof}
    For each $h \ge 0$, define $g_h := \gcd(k-h, s_1, \dots, s_r)$.
    Consider the formula in Proposition~\ref{prop: H^*[1] coeff}. Let $h \ge 0$ and $I \in \mathcal I_h$. Assume that the product of binomials $\binom{\lambda_1}{I_1} \cdots \binom{\lambda_{k-1}}{I_{k-1}}$ is nonzero. For each $s \in [k-1]$ such that $\lambda_s \ge 1$, we have that $g | s$. Therefore each nonzero term of $1 \cdot I_1 + \dots + (k-1) \cdot I_{k-1}$ is divisible by $g$, hence $g$ divides $h$ and so $d(g,h) = 1$. Since $g$ divides $h$, we have that $g_h = \gcd(k-h,s_1, \dots, s_r) = \gcd(k, s_1, \dots, s_r) = g$. The result immediately follows. 
\end{proof}

\subsection{Counting non-hypersimplicial DOSPs}\label{sec: non-hyp dosps}

In this section we count the non-hypersimplicial $\sigma$-fixed $(k,n)$-DOSPs. 
Throughout this section, we will frequently make use of the alternative definition of DOSPs in terms of functions $[n] \rightarrow \ZZ/k\ZZ$. See the beginning of Section~\ref{sec: DOSP thm} and Definition~\ref{def: dist in dosp and winding no}.

\medskip

\noindent \textbf{Setup for counting non-hypersimplicial DOSPs.}
Fix $k < n$ and $\sigma\in S_n$. We denote by $\MD$ the set of $\sigma$-fixed non-hypersimplicial DOSPs and by $\MD^\tau\subseteq\MD$ the subset of DOSPs with turning number $\tau \in \ZZ/k\ZZ$.
We define the set $\Lambda$ of non-empty unions of cycles in $\sigma$:
\[
    \Lambda = \set{C_{i_1}\cup C_{i_2}\cup \cdots\cup C_{i_s}\colon \set{i_1, i_2,\ldots, i_s}\subseteq [r], s > 0}.
\]
For each $u=C_{i_1}\cup C_{i_2}\cup \cdots\cup C_{i_s} \in \Lambda$, we will denote the corresponding set $\set{i_1, i_2,\ldots, i_s}\subseteq [r]$ by $\ind(u)$.
Furthermore, we define the subset $\MD_u^{\tau}\subseteq\MD^{\tau}$ of DOSPs containing a tuple $(L,\ell)$ such that:
\begin{itemize}
    \item $\abs{L}\leq \ell$ (we call such $L$ a \emph{bad set}),
    \item $L$ completely lies in $u$, and
    \item for every $C_i\subseteq u$, $L\cap C_i$ is non-empty.
\end{itemize}
In other words, $\MD^\tau_u$ is the set of all $\sigma$-fixed DOSPs $D$ such that $u$ is a disjoint union of bad sets of $D$, and those bad sets form a single $\sigma$ orbit. 
Note, for any $D\in\MD^\tau$ there exists $u \in \Lambda$ such that $D\in\MD_u^{\tau}$. Explicitly, $D$ is non-hypersimplicial so it has a bad set, say $(L, \ell)$, then $D \in \MD_u^\tau$ where $u = L \cup \sigma(L) \cup \sigma^2(L) \cup \dots \cup \sigma^{o(\sigma)-1}(L)$ is the $\sigma$-orbit of $L$.

\begin{lemma}\label{lem: count SOME non-hyp DOSPs}
    Fix $2\leq k < n$ and let $\sigma\in S_n$ be a permutation with cycle type $(s_1,s_2,\dots,s_r)$.
    Define $g=\gcd(s_1,\ldots,s_r,k)$ and let $\tau\in\ZZ/k\ZZ$ such that $g\cdot\tau = 0$. Fix $h \in [k-1]$ and let $J\in\binom{\Lambda}{h}$ be a non-empty subset of $\Lambda$ such that the elements of $J$ are pairwise disjoint.
    Then
    \[
        \abs*{\bigcap_{u\in J}\MD^\tau_u} =
        \frac{((k-i)/o(\tau)+h-1)!}{((k-i)/o(\tau))!}
        o(\tau)^{j-1}
        (k-i)^{r-j}
    \]
    where $o(\tau)$ denotes the order of $\tau$,
    $j$ is the number of elements in $\bigcup_{u\in J}\ind(u)$,
    and $i$ is the number of elements in $\bigcup_{u\in J} $u.
\end{lemma}

Before giving the proof, we will outline the concepts in the proof with an example.

\begin{example}\label{ex: non-hyp dosp}
    Fix $k=12$, $n=24$, and $\sigma\in S_n$ with cycle type $(3,3,6,3,9)$.
    This means that $r=5$ and $g=3$.
    For simplicity, we assume that the cycle sets are
    \[
    C_1=\set{1,2,3},\ 
    C_2=\set{4,5,6},\ 
    C_3=\set{7,8,\ldots,12},\ 
    C_4=\set{13,14,15},\ 
    C_5=\set{16,17,\ldots,24}. 
    \]
    For each cycle $C_i$, we fix a distinguished element $q_i \in C_i$. Explicitly, we choose $q_i$ to be the smallest element: $q_1=1$, $q_2=4$, $q_3=7$ etc.
    Let us fix a subset $J = \{ u_1, u_2\}$ where $u_1=C_1\cup C_2$ and $u_2=C_4$.
    This gives us 
    \[
    \bigcup_{u \in J} u = \{1,2,3, 4, 5, 6, 13, 14, 15\}
    \quad \text{and} \quad
    \bigcup_{u \in J} \ind(u) = \{1,2,4\},
    \] hence $i=9$ and $j=3$.
    Lastly we fix $\tau=8 \in \ZZ/12\ZZ$ meaning that $o(\tau)=3$.
    We will give an overview of the proof of Lemma~\ref{lem: count SOME non-hyp DOSPs}, which counts the number of non-hypersimplicial DOSPs in $\MD^\tau_{u_1} \cap \MD^\tau_{u_2}$. To do this, we construct DOSPs in this set. We imagine starting with an \emph{empty DOSP} $(L_1 = \{ \}, \dots, L_{12} = \{ \})$ of $k = 12$ empty sets. We then consider the possible ways to place the cycles into the DOSP. Note that the turning number $\tau = 8$ is fixed, so each $\sigma$-orbit consists of $o(\tau) = 3$ sets of the DOSP. We will place the cycles into the DOSP with three steps.
    
    Our first step is to distribute $u_1$ across a single $\sigma$-orbit of $o(\tau)=3$ sets of $2$ elements each and to adorn each of these three sets with a decoration $\ell_i \geq 2$ so that each set is a bad set of the resulting DOSP.
    Our second step is to distribute $u_2$ across $3$ singletons, which we note always result in bad sets in the final DOSP.
    The third step is to put the rest of the elements into the remaining spaces.
    See Figure~\ref{fig: example dosp} for a specific instance.

    \medskip
    \noindent \textbf{Step 1.}
    A $\sigma$-fixed DOSP $D$ with turning number $\tau$ is completely determined by the values the function $f_D$ takes on the distinguished elements $q_i$.
    In Figure~\ref{fig: example dosp}, the $q_i$ are the underlined elements.
    As a starting point, we will assume that $f_D(q_1) = 0$, or in other words $1 \in L_1$. This choice fixes the positions of the elements in $C_1 = \{1,2,3\}$.
    Since $\tau = 8$, it follows that $f_D(u_1)=\set{0,4,8}$, meaning that we have $o(\tau)=3$ choices for the position of $4=q_2\in u_1$. Once we have placed $u = C_1 \cup C_2$, we mark the positions $0,1, 4,5, 8,9$ as \emph{filled}, this guarantees that each set in the resulting DOSP containing elements of $u_1$ are bad sets. In Figure~\ref{fig: example dosp}, these filled sets include the white circles.

    \medskip
    \noindent \textbf{Step 2.}
    The placement of the element $13=q_4\in u_2$ is restricted to the locations $\set{2,6,10}$ and $\set{3,7,11}$ because $\set{1,5,9}$ (the white spaces in Figure~\ref{fig: example dosp}) need to remain clear.
    This gives us $6$ choices for $q_4$.
    Notice that we count possible locations for a $q_i$ in sets of $3$.
    This corresponds to the factor $o(\tau)^{j-1}$ in the formula in the lemma.
    
    \medskip
    \noindent \textbf{Step 3.}
    Finally, we must choose placements for the remaining cycles $C_3$ and $C_5$. After having placed $u_1$ and $u_2$, there are only $3$ spaced left in the DOSP. So, we have $3$ choices for $q_3=7$ and $q_5=16$, which corresponds to $9$ choices to finish off the DOSP.
    This part corresponds to the right-most factor $(k-i)^{r-j}$ in the formula from the lemma.
\end{example}

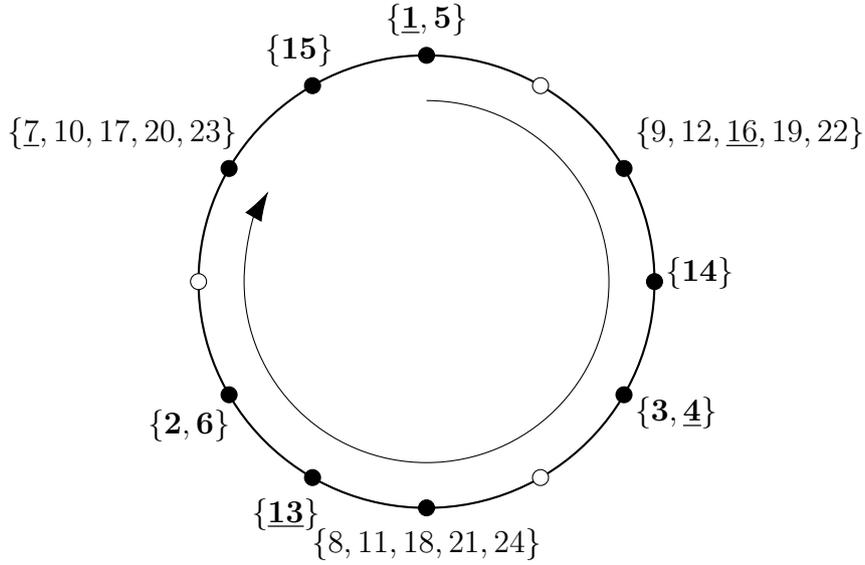
\begin{figure}
\centering
\begin{tikzpicture}[scale = 0.8]

\draw[thick] (6,12.25) circle (3.75);

\vertex [label={$\mathbf{\{ \underline{1}, 5 \}}$}] at ({6+3.75*cos(90)},{12.25+3.75*sin(90)}) {};
\smallvertex [] at ({6+3.75*cos(60)},{12.25+3.75*sin(60)}) {};
\vertex [label={[anchor=south west,] $\{ 9, 12, \underline{16}, 19, 22 \}$}] at ({6+3.75*cos(30)},{12.25+3.75*sin(30)}) {};
\vertex [label={[anchor=west,] $\mathbf{\{ 14 \}}$}]
at ({6+3.75*cos(0)},{12.25+3.75*sin(0)}) {};
\vertex [label={[anchor=north west,] $\mathbf{\{ 3, \underline{4} \}}$}]
at ({6+3.75*cos(-30)},{12.25+3.75*sin(-30)}) {};
\smallvertex []
at ({6+3.75*cos(-60)},{12.25+3.75*sin(-60)}) {};
\vertex [label={[yshift=-27] $\{ 8, 11, 18, 21, 24 \}$}]
at ({6+3.75*cos(-90)},{12.25+3.75*sin(-90)}) {};
\vertex [label={[xshift=-10, yshift=-27]$\mathbf{\{ \underline{13} \}}$}]
at ({6+3.75*cos(-120)},{12.25+3.75*sin(-120)}) {};
\vertex [label={[xshift=-15, yshift=-25]$\mathbf{\{ 2, 6 \}}$}]
at ({6+3.75*cos(-150)},{12.25+3.75*sin(-150)}) {};
\smallvertex []
at ({6+3.75*cos(-180)},{12.25+3.75*sin(-180)}) {};
\vertex [label={[xshift=-40]$\{ \underline{7}, 10, 17, 20, 23 \}$}]
at ({6+3.75*cos(-210)},{12.25+3.75*sin(-210)}) {};
\vertex [label={[xshift=-5]$\mathbf{\{ 15 \}}$}]
at ({6+3.75*cos(-240)},{12.25+3.75*sin(-240)}) {};

\draw[-{Latex[length=4mm, width=2.5mm]}] ({6+3*cos(90)},{12.25+3*sin(90)}) arc[start angle=90, end angle=-210, radius=3];

\end{tikzpicture}

\caption{A DOSP with the setup of Example~\ref{ex: non-hyp dosp}, the choices for the $q_i$ are $f_D(1)=0$, $f_D(4)=4$, $f_D(7)=10$ $f_D(13)=7$, $f_D(16)=3$, are the underlined elements. The bold sets are the bad sets corresponding to $u_1$ and $u_2$}
\label{fig: example dosp}
\end{figure}

\begin{proof}[Proof of Lemma~\ref{lem: count SOME non-hyp DOSPs}]
    For the purpose of this proof, we will endow the set $\ZZ/k\ZZ$ with a total ordering induced by identifying it with the set $\set{0,1,\ldots,k-1}$.
    We write the elements of $J$ as $u_1,u_2,\ldots,u_h$.
    Let 
    \[
    D = ((L_1, \ell_1), (L_2, \ell_2), \dots, (L_t, \ell_t)) \in \bigcap_{u \in J} \MD^\tau_u
    \]
    be a DOSP.
    Without loss of generality, we may assume that $L_1 \subseteq u_1$.
    Since the turning number of $D$ is $\tau$, we may assume that $f_D(L_1) = 0$ and $f_D(\sigma(L_1)) = \tau$.
    For each $u_a$ with $a \in \{2, \dots h\}$, there exists a set $L$ of $D$ such that $0 < f_D(L) < \tau$ with $L \subseteq u_a $ and we write $p_a = f_D(L)$ for the value of this function on $L$.
    We also fix $p_1=0$.
    Let us count the DOSPs $D$, as above, such that $p_2 < \dots < p_h$.

    Suppose that we are given the values of function $f_D(L)$ for each set $L \subseteq u_a$ over all $u_a$.
    Let us count the ways to distribute the elements of the $u_a$ into the DOSP if $u_a$ is already placed.
    For each $b \in \ind(u_a)$, there are $k/o(\tau)$ different possible values for $f_D(q_b)$.
    So, in total, there are $o(\tau)^{j-1}$ possible choices for the positions of the $q$'s, where the first $q$ is put in the first position and the other $q$'s, of which there are $j-1$, are placed relatively to the first. 
    
    Now let us count the ways to position the sets $u_1, \dots, u_h$ in the DOSP. 
    The position of $u_1$ and its corresponding sets of the DOSP is fixed.
    Then we use a stars and bars argument to count the  placements the remaining spaces between the $u_a$'s:
    \[
    (L(u_1)) \ \square \ (L(u_2)) \ \square \ (L(u_3)) \ \square  \ \dots \ \square \ (L(u_h)) \ \square,
    \]
    where $L(u_a)$ is the set of the DOSP with $f_D(L(u_a))=p_a$ and the boxes represent some number of spaces.
    Since each set $L(u_a)$ is a bad set, it takes up at least $|L(u_a)|$ spaces of the DOSP.
    Since $u_a$ is a $\sigma$-orbit of bad sets, we have that each set that partitions $u_a$ is bad and together they take up $|u_a|$ spaces of the DOSP.
    So the bad sets, whose union is $u_1 \cup \dots \cup u_h$, take up $i$ spaces of the DOSP.
    Hence, in total there are $k-i$ spaces to place between the $u_a$'s.
    Note that whenever we place one space, its $\sigma$-orbit has length $o(\tau)$.
    So we are free to choose the positions of $(k-i)/o(\tau)$ spaces and the rest are determined by their $\sigma$-orbits.
    So by stars and bars there are $\binom{(k-i)/o(\tau) + h-1}{(k-i)/o(\tau)} = \binom{(k-i)/o(\tau) + h-1}{h-1}$ many ways place the spaces.

    So far, we have fixed the position of all cycles of $\sigma$ in $u_1, \dots, u_h$, i.e., we have placed $j$ cycles into the DOSP.
    There are $r-j$ remaining cycles.
    Placing a cycle $C_c$ into the DOSP is equivalent to choosing the position of $q_c$.
    Each $q_c$ can be placed into the $k-i$ spaces.
    Hence, there are $(k-i)^{r-j}$ DOSPs with the given $u_a$ positions.

    Finally, there are $(h-1)!$ different total orderings of $p_2,  \dots, p_h$ and each gives the same number of DOSPs.
    So, the number of $\sigma$-fixed non-hypersimplicial DOSPs with turning number $\tau$ is:
    \[
    \binom{(k-i)/o(\tau)+h-1}{h-1}(h-1)!o(\tau)^{j-1}(k-i)^{r-j} 
    = 
    \frac{((k-i)/o(\tau)+h-1)!}{((k-i)/o(\tau))!}o(\tau)^{j-1}(k-i)^{r-j}.
    \]
\end{proof}

We proceed to count the $\sigma$-fixed non-hypersimplicial DOSPs.

\begin{lemma}\label{lem: count non-hyp DOSPs}
    Fix $2\leq k < n$ and let $\sigma\in S_n$ be a permutation with cycle type $(s_1,s_2,\dots,s_r)$.
    For each $i\in [n]$, let $\lambda_i$ be the number of length $i$-cycles of $\sigma$.
    Let $g=\gcd(s_1,\ldots,s_r,k)$ and define the set $T=\set{\tau\in\ZZ/k\ZZ\colon g\cdot\tau=0}$.
    The number of $\sigma$-fixed non-hypersimplicial DOSPs is
    \begin{equation}\label{eq: non-hypersimplicial dosps}
    \sum_{\tau \in T} 
    \sum_{h = 1}^{k-1} 
        (-1)^{h+1} 
        \sum_{i = h}^{k-1} 
        \sum_{j = 1}^i 
        \frac{((k-i)/o(\tau)+h-1)!}{((k-i)/o(\tau))!}
        o(\tau)^{j-1}
        (k-i)^{r-j}
        \left(
    \sum_{\substack{I \in \mathcal I_i \\ |I| = j}}
    \binom{\lambda_1}{I_1} \cdots \binom{\lambda_{k-1}}{I_{k-1}}
    \right)
    \stirling{j}{h}.
    \end{equation}
\end{lemma}

\begin{proof}
    Let $\tau\in\ZZ/k\ZZ$ satisfy $g\cdot\tau = 0$ and let $\MD^{\tau}$ be the set of non-hypersimplicial $\sigma$-fixed DOSPs with turning number $\tau$.
    The number of $\sigma$-fixed non-hypersimplicial DOSPs is
    \[
    \sum_{\tau\in T} \abs{\MD^{\tau}}.
    \]
    By the inclusion-exclusion principle, we have
    \[
    \abs*{\MD^{\tau}} = 
    \abs*{\bigcup_{u\in\Lambda}\MD_u^{\tau}} = 
    \sum_{h\geq 1} (-1)^{h+1} \sum_{J\in\binom{\Lambda}{h}}\abs*{\bigcap_{u\in J}\MD_u^{\tau}}.
    \]
    Given a non-empty subset $J\subseteq\Lambda$, suppose we have $u_1,u_2\in J$.
    For a DOSP $D$ to lie both in $\MD_{u_1}$ and $\MD_{u_2}$, it means there exist (not necessarily distinct) sets $L_1$ and $L_2$ whose $\sigma$-orbits are $u_1$ and $u_2$ respectively.
    In particular, if $u_1\cap u_2\neq\emptyset$, the $\sigma$-orbits must be the same and $u_1=u_2$.
    Hence, we may always assume that the sets $u_i$ contained in $J$ are disjoint.
    It also follows that $h$ is bounded by the number of sets $L_i$, which is $k$.
    In the case $h=k$, every $L_i$ satisfies $\abs{L_i} = \ell_i = 1$, which means that $n=k$, a contradiction.
    Thus $1\leq h\leq k-1$.
    We define the set
    \[
        \Lambda(h,i,j) = \set*{ J\in\binom{\Lambda}{h} : \abs*{\bigcup_{u\in J}u} = i\text{ and } \abs*{\bigcup_{u\in J}\ind(u)} = j }
    \]
    This is the collection of all $h$-subsets of $\Lambda$ involving exactly $j$ distinct cycles that contain a total of $i$ elements across all cycles.
    The cardinality of $\Lambda(h,i,j)$ is exactly
    \[
        \left(
            \sum_{\substack{I\in\mathcal{I}_i \\ \abs{I} = j}}
            \binom{\lambda_1}{I_1} \cdots \binom{\lambda_{k-1}}{I_{k-1}}
        \right)\stirling{j}{h},
    \]
    which follows from the argument that there are $\stirling{j}{h}$ ways to partition $j$ distinct cycles into $h$ sets and that we choose $I_1$ many fixed points of $\sigma$, $I_2$ many $2$-cycles of $\sigma$, and so on, such that $\abs{I} = j$ and $1\cdot I_1 + 2\cdot I_2 +\cdots + (k-1)\cdot I_{k-1} = i$.
    With this, we apply Lemma~\ref{lem: count SOME non-hyp DOSPs} and rearrange the previous formula:
    \begin{align*}
    \abs*{\MD^{\tau}}
    & = 
    \sum_{h=1}^{k-1}
    (-1)^{h+1}\sum_{i=h}^{k-1}
    \sum_{j=1}^r
    \sum_{J\in\Lambda(h,i,j)} \abs*{\bigcap_{u\in J}\MD_u^{\tau}} \\
    & = \sum_{h=1}^{k-1}
    (-1)^{h+1}\sum_{i=h}^{k-1}
    \sum_{j=1}^r
    \sum_{J\in\Lambda(h,i,j)}
    \frac{((k-i)/o(\tau)+h-1)!}{((k-i)/o(\tau)!}
        o(\tau)^{j-1}
        (k-i)^{r-j} \\
    & = \sum_{h=1}^{k-1}
    (-1)^{h+1}\sum_{i=h}^{k-1}
    \sum_{j=1}^r
    \frac{((k-i)/o(\tau)+h-1)!}{((k-i)/o(\tau)!}
        o(\tau)^{j-1}
        (k-i)^{r-j} 
    \sum_{J\in\Lambda(h,i,j)} 1 \\
    & = \sum_{h=1}^{k-1}
    (-1)^{h+1}\sum_{i=h}^{k-1}
    \sum_{j=1}^r
    \frac{((k-i)/o(\tau)+h-1)!}{((k-i)/o(\tau)!}
        o(\tau)^{j-1}
        (k-i)^{r-j} 
    \left(
        \sum_{\substack{I\in\mathcal{I}_i \\ \abs{I} = j}}
        \binom{\lambda_1}{I_1} \cdots \binom{\lambda_{k-1}}{I_{k-1}}
    \right)\stirling{j}{h}.
    \end{align*}
    
\end{proof}

\subsection{Proof of main result}
In this section, we simplify the formula in Lemma~\ref{lem: count non-hyp DOSPs} to prove Theorem~\ref{thm: H*[1] is perm character of DOSPs under Sn}. We use the following identity derived from the falling factorial identity for Stirling numbers. For each $j \ge 1$, define the Laurent polynomial $F_j(y) \in \QQ[y, y^{-1}]$ by
    \[
    F_j(y) := 
    \left(
    \frac{1}{y}
    \right)^{j-1}
    \sum_{h = 1}^j 
    (-1)^{h+1} \stirling{j}{h} (y+1)(y+2) \cdots (y+h-1). 
    \]
These Laurent polynomials are, in fact, constants.

\begin{lemma}\label{lem: stirling identity}
    We have $F_j(y) = (-1)^{j+1}$.
\end{lemma}

\begin{proof}
    For each $h \ge 0$, notice that
    $
        (-1)^{h+1} (y+1)(y+2)\cdots (y+h-1) = \frac{(-y)_h}{y}.
    $
    By Proposition~\ref{prop: stirling2}, with $x = -y$, we obtain $F_j(y)=\frac{(-y)^{j-1}}{y^{j-1}}$, which concludes the proof.
\end{proof}

\begin{proof}[Proof of Theorem~\ref{thm: H*[1] is perm character of DOSPs under Sn}]
    By Corollary~\ref{cor: H*[1] coeff short}, we have
    \[
    g k^{r-1} - H^*[1](\sigma) =
        -g\sum_{h = 1}^{k-1}
    \left(
    \sum_{I \in \mathcal I_i}
    (-1)^{|I|}
    \binom{\lambda_1}{I_1}
    \cdots
    \binom{\lambda_{k-1}}{I_{k-1}}
    \right)
    (k - i)^{r-1}.
    \]
    We will show that the expression above is equal to the number of $\sigma$-fixed non-hypersimplicial $(k,n)$-DOSPs.
    Define the set $T = \set{\tau\in\ZZ/k\ZZ : g \tau = 0}$ and note that $|T| = g$.
    By Lemma~\ref{lem: count non-hyp DOSPs}, we have that the number of $\sigma$-fixed non-hypersimplicial $(k,n)$-DOSPs is
    \[
    \sum_{\tau \in T} 
    \sum_{h = 1}^{k-1} 
        (-1)^{h+1} 
        \sum_{i = h}^{k-1} 
        \sum_{j = 1}^i 
        \frac{((k-i)/o(\tau)+h-1)!}{((k-i)/o(\tau))!}
        o(\tau)^{j-1}
        (k-i)^{r-j}
        \left(
    \sum_{\substack{I \in \mathcal I_i \\ |I| = j}}
    \binom{\lambda_1}{I_1} \cdots \binom{\lambda_{k-1}}{I_{k-1}}
    \right)
    \stirling{j}{h}.
    \]
We reorder the sums in this expression to obtain
\begin{align*}
& \sum_{\tau \in T} 
    \sum_{h = 1}^{k-1} 
        (-1)^{h+1} 
        \sum_{i = h}^{k-1} 
        \sum_{j = 1}^i 
        \frac{((k-i)/o(\tau)+h-1)!}{((k-i)/o(\tau))!}
        o(\tau)^{j-1}
        (k-i)^{r-j}
        \left(
    \sum_{\substack{I \in \mathcal I_i \\ |I| = j}}
    \binom{\lambda_1}{I_1} \cdots \binom{\lambda_{k-1}}{I_{k-1}}
    \right)
    \stirling{j}{h} \\
=& 
    \sum_{\tau \in T}
    \sum_{i = 1}^{k-1} 
    \sum_{j = 1}^{i}
    \sum_{h = 1}^{i} 
        (-1)^{h+1} 
        \frac{((k-i)/o(\tau)+h-1)!}{((k-i)/o(\tau))!}
        o(\tau)^{j-1}
        (k-i)^{r-j}
        \left(
    \sum_{\substack{I \in \mathcal I_i \\ |I| = j}}
    \binom{\lambda_1}{I_1} \cdots \binom{\lambda_{k-1}}{I_{k-1}}
    \right)
    \stirling{j}{h} \\
=& 
    \sum_{\tau \in T}
    \sum_{i = 1}^{k-1} 
    \sum_{j = 1}^{i}
        \left(
        \sum_{\substack{I \in \mathcal I_i \\ |I| = j}}
        \binom{\lambda_1}{I_1} \cdots \binom{\lambda_{k-1}}{I_{k-1}}
        \right)
        (k-i)^{r-j}
        o(\tau)^{j-1}
    \sum_{h = 1}^{i} 
        (-1)^{h+1} 
        \frac{((k-i)/o(\tau)+h-1)!}{((k-i)/o(\tau))!}
    \stirling{j}{h}.
\end{align*}
Next, we apply Lemma~\ref{lem: stirling identity} to the above expression by setting $y = (k-i)/o(\tau)$. Note that $o(\tau)^{j-1} = (k-i)^{j-1} (1/y)^{j-1}$. So, the above expression is equal to the following
\begin{align*}
&
    \sum_{\tau \in T}
    \sum_{i = 1}^{k-1} 
    \sum_{j = 1}^{i}
        \left(
        \sum_{\substack{I \in \mathcal I_i \\ |I| = j}}
        \binom{\lambda_1}{I_1} \cdots \binom{\lambda_{k-1}}{I_{k-1}}
        \right)
        (k-i)^{r-j}
        (k-i)^{j-1}
        F_j((k-i)/o(\tau))\\
=&
    \sum_{\tau \in T}
    \sum_{i = 1}^{k-1} 
    \sum_{j = 1}^{i}
        \left(
        \sum_{\substack{I \in \mathcal I_i \\ |I| = j}}
        \binom{\lambda_1}{I_1} \cdots \binom{\lambda_{k-1}}{I_{k-1}}
        \right)
        (k-i)^{r-1}
        (-1)^{j+1}\\
=&
    -g
    \sum_{i = 1}^{k-1} 
        \left(
        \sum_{I \in \mathcal I_i}
        (-1)^{|I|}
        \binom{\lambda_1}{I_1} \cdots \binom{\lambda_{k-1}}{I_{k-1}}
        \right)
        (k-i)^{r-1} = gk^{r-1} - H^*[1](\sigma).
\end{align*}
So, the number of $\sigma$-fixed non-hypersimplicial DOSPs is equal to $gk^{r-1} - H^*[1](\sigma)$. By Corollary~\ref{cor: size of Phi}, the number of $\sigma$-fixed $(k,n)$-DOSPs is $gk^{r-1}$. Therefore $H^*[1](\sigma)$ is equal to the number of $\sigma$-fixed hypersimplicial $(k,n)$-DOSPs. 
This completes the proof.
\end{proof}

\subsection{Recurrence relation}
\label{sec: recurrence relation}

In this section, we show that $H^*(\Delta_{k,n}; S_n)[1](\sigma)$ satisfies a recurrence relation similar to that for Eulerian numbers. Given $k \in \ZZ$, a tuple $\lambda = (\lambda_1, \lambda_2, \dots, \lambda_{n}) \in \ZZ^{n}_{\ge 0}$, and $r \ge 1$, we define
\[
B(k, \lambda, r) = g(k, \lambda) \sum_{h = 0}^{k-1}
\left(
\sum_{I \in \mathcal I_h}
(-1)^{|I|}
\binom{\lambda_1}{I_1}
\cdots
\binom{\lambda_{k-1}}{I_{k-1}}
\right)
(k - h)^{r-1} 
\]
where $g(k, \lambda) = \gcd(\{k\} \cup \{i \in [n] : \lambda_i \ge 1\})$.

Suppose that $\sigma \in S_n$ has cycle type $(s_1, \dots, s_r)$ and for each $i \in [n]$ we denote by $\lambda_i$ the number of cycles of $\sigma$ of length $i$. Then, by Theorem~\ref{thm: H*[1] is perm character of DOSPs under Sn}, we have $H^*(\Delta_{k,n}; S_n)[1](\sigma) = B(k, \lambda, r)$. 

\begin{proposition}\label{prop: recurrence relation}
    We have $B(k, \lambda, r) = 0$ if $k < 1$, $B(1, \lambda, r) = gk^{r-1}$, and
    $B(k, \lambda, r) = gk^{r-1}$ if $\lambda_1 = \cdots = \lambda_{k-1} = 0$.
    Suppose there exists $a \in [k-1]$ such that $\lambda_a \ge 1$. Define 
    $
    \lambda' = (\lambda_1, \dots, \lambda_{a-1}, \lambda_a - 1, \lambda_{a+1}, \dots, \lambda_n).
    $ Then, we have the recurrence relation
    \[
    B(k, \lambda, r) = \frac{g(k, \lambda)}{g(k, \lambda')} B(k, \lambda', r)
    - 
    \frac{g(k, \lambda)}{g(k-a,\lambda')} B(k-a, \lambda', r).
    \]
\end{proposition}

\begin{proof}
The first part of the result follows easily from the definition of $B(k, \lambda, r)$. For the recurrence relation, fix $a \in [k-1]$ such that $\lambda_a \ge 1$. To simplify notation, we write $g = g(k, \lambda)$, $g' = g(k, \lambda')$, and $g'' = g(k-a, \lambda')$. First, we apply Pascal's identity
\begin{align*}
B(k, \lambda, r) &= 
g\sum_{h = 0}^{k-1}
\left(
\sum_{I \in \mathcal I_h}
(-1)^{|I|}
\binom{\lambda_1}{I_1}
\cdots
\binom{\lambda_{k-1}}{I_{k-1}}
\right)
(k - h)^{r-1} \\
&=
g\sum_{h = 0}^{k-1}
\left(
\sum_{I \in \mathcal I_h}
(-1)^{|I|}
\binom{\lambda_1}{I_1}\cdots
\binom{\lambda_a - 1}{I_a} \cdots
\binom{\lambda_{k-1}}{I_{k-1}}
\right)
(k - h)^{r-1} \\
& \quad
+g\sum_{h = 0}^{k-1}
\left(
\sum_{I \in \mathcal I_h}
(-1)^{|I|}
\binom{\lambda_1}{I_1}\cdots
\binom{\lambda_a - 1}{I_a - 1} \cdots
\binom{\lambda_{k-1}}{I_{k-1}}
\right)
(k - h)^{r-1}.
\end{align*}

The first sum coincides with $(g/g') B(k, \lambda', r)$. For the second sum, we re-index as follows
\begin{align*}
B(k, \lambda, r) - \frac{g}{g'} B(k, \lambda', r) &= 
g\sum_{h = 0}^{k-1}
\left(
\sum_{I \in \mathcal I_h}
(-1)^{|I|}
\binom{\lambda_1}{I_1}\cdots
\binom{\lambda_a - 1}{I_a - 1} \cdots
\binom{\lambda_{k-1}}{I_{k-1}}
\right)
(k - h)^{r-1} \\
&=
g\sum_{h = 0}^{k-1-a}
\left(
\sum_{I \in \mathcal I_{h+a}}
(-1)^{|I|}
\binom{\lambda'_1}{I_1}\cdots
\binom{\lambda'_a}{I_a - 1} \cdots
\binom{\lambda'_{k-1}}{I_{k-1}}
\right)
(k - h - a)^{r-1} \\
&= 
g\sum_{h = 0}^{(k-a)-1}
\left(
\sum_{I \in \mathcal I_{h}}
(-1)^{|I|+1}
\binom{\lambda'_1}{I_1}\cdots
\binom{\lambda'_{k-1}}{I_{k-1}}
\right)
((k - a) - h)^{r-1} \\
&= - \frac{g}{g''} B(k-a, \lambda', r).
\end{align*}
So we have shown that $B$ satisfies the recursive relation, which concludes the proof.
\end{proof}

\begin{remark}
    The evaluation of $H^*(\Delta_{k,n}; S_n)[1]$ at the identity is equal to Eulerian number $A(n-1, k-1)$, which is equal to $B(k, (n,0, \dots, 0), n)$. The standard recurrence relation for  Eulerian numbers is
    \[
    A(n,k) = (k+1)A(n-1,k) + (n-k)A(n-1,k-1).
    \]
    The equation that we obtain, using Proposition~\ref{prop: recurrence relation}, by starting with an Eulerian number is
    \[
    B(k,(n,0, \dots, 0), n) = B(k, (n-1,0, \dots, 0), n) - B(k-1, (n-1, 0, \dots, 0), n).
    \]
    This relation differs a little from the standard recurrence as the terms appearing on the right-hand side are, in general, not Eulerian numbers.
\end{remark}

\section{The second hypersimplex}\label{sec: k=2}

In this section, we give a complete description of the coefficients of the $H^*$-polynomial for the second hypersimplex $\Delta_{2,n}$. Throughout this section, we denote $H^* = H^*(\Delta_{2,n}, S_n)$. We interpret the coefficients $H^*_i$ of $H^*$ in terms of DOSPs, see Definition~\ref{def: dosp}, as well as actions of $S_n$ on subsets and partitions of $[n]$.

\medskip
\noindent \textbf{Notation.} For each $m \in [n]$, we denote by $\rho_m$ the character of the permutation representation of $S_n$ acting on $\binom{[n]}{m}$. So, we have $\rho_m(\sigma) = |\{S \subseteq [n] : |S| = m,\, \sigma(S) = S\}|$ for each $\sigma \in S_n$. By taking complements, we have $\rho_{n-m} = \rho_m$ for each $m$. For each $\sigma \in S_n$ and $i \in [n]$, we write $\lambda_i$ for the number of cycles of $\sigma$ of length $i$. So, we have $\rho_{n-1} = \rho_1 = \Chi_{\rm nat} = \lambda_1$ is the character of the \emph{natural representation} and $\rho_n = \Chi_0$ is the trivial character. We define $\tau_m$ to be the character of the permutation representation of $S_n$ acting on the set of partitions of $[n]$ into two parts: one of size $m$ and the other of size $n-m$. Note that, unless $n$ is even and $m = n/2$, we have that $\rho_m = \tau_m$.

\begin{theorem}\label{thm: k=2 H* coeff}
    Let $n > 2$. The coefficients of the $H^*$-polynomial $H^*(\Delta_{2,n}; S_n)$ are:
    \begin{itemize}
        \item $H^*_0 = \Chi_0$ the trivial character,
        \item $H^*_1 = \rho_2 - \rho_1$,
        \item $H^*_m = \rho_{2m}$ for each $2 \le m \le \lfloor n/2 \rfloor$.
    \end{itemize}
    The evaluation of $H^*$ at one is given by $H^*[1] = \Chi_0 + \tau_2 + \tau_3 + \dots + \tau_{\lfloor n/2 \rfloor}$, which is a permutation character.
    If $n$ is odd, then the leading coefficient is $H^*_{(n-1)/2} = \rho_1$. Otherwise, if $n$ is even, then the leading coefficient is $H^*_{n/2} = \Chi_0$.
\end{theorem}

Before we prove the theorem, we consider the formula for the coefficients $H^*_m$ in Theorem~\ref{thm: H* coefficient in terms of representations}. We note that the final term, indexed by $h = k-1$, has a simple description.

\begin{proposition} \label{cor: subset for ell-h=1}
    Let $\sigma \in S_n$ and $m \ge 0$. We have
    \[
    |\Phi_1(\sigma, m)| = \begin{cases}
        1 & \text{if } m = 0, \\
        0 & \text{otherwise}.
    \end{cases}
    \]
\end{proposition}

\begin{proof}
    The set $\Phi_1(\sigma, m)$ consists of functions $f \colon [r] \rightarrow \{0\}$ such that $\sum_{i = 1}^r f(i) s_i = m$. There is only one such function, which belongs to the set $\Phi_1(\sigma, 0)$.
\end{proof}

To prove Theorem~\ref{thm: k=2 H* coeff}, we require the following result about $S_n$-representations.

\begin{lemma}\label{lem: equal reps for n even}
    Let $n$ be even. The following equation of $S_n$-representations holds
    \[
    \sum_{m = 0}^{n/2}
    \rho_{2m}
    = \sum_{m = 0}^{n/2} \tau_m.
    \]
\end{lemma}

\begin{proof}
    Fix a permutation $\sigma$ with cycle type $(s_1, s_2, \dots, s_r)$. It suffices to show that the number of subsets of $[n]$ with even size that are fixed by $\sigma$ is equal to the number of two-part partitions of $[n]$ that are fixed by $\sigma$. Suppose that $A \sqcup B = [n]$ is a two-part partition, then we write $AB := \{A, B\}$ for the partition of $[n]$ into $A$ and $B$. Given a partition $AB$ of $[n]$, we write $\sigma(AB)$ for the partition of $[n]$ with parts $\sigma(A)$ and $\sigma(B)$. We define the sets
    \[L^\sigma = \{A \subseteq [n] : |A| \text{ is even},\, \sigma(A) = A \}
    \text{ and }
    R^\sigma = \{AB : A \sqcup B = [n],\, \sigma(AB) = AB \}.\]
    We will show that $|L^\sigma| = |R^\sigma|$. 
    
    We first consider the special case where each cycle of $\sigma$ has odd length. Each element $A \in L^\sigma$ is the union of the supports of cycles of $\sigma$. Since $|A|$ is even and each cycle has odd length, it follow that $A$ is the union of an even number of cycle supports. The indices of the cycles whose supports form $A$ uniquely determine $A$, and any even subset of cycles forms a unique set $A$. So we have
    \begin{align*}
        |L^\sigma| &= |\{S \subseteq [r] : \Sigma S \text{ is even} \}| \\
        &= |\{S \subseteq [r] : |S| \text{ is even} \}|\\
        &= 2^{r-1}.
    \end{align*}
    On the other hand, for each subset $S \subset [r]$, we obtain a partition $TU$ where $T$ is the union of supports of the cycles in $\sigma$ indexed by $S$ and $U = [n] \setminus T$ is the complement of $T$. Observe that every $\sigma$-invariant partition arises in this way because each cycle has odd length. Moreover, each partition arises from a subset $S \subseteq [k]$ or its complement $[k] \setminus S$. So we conclude that $|R^\sigma| = 2^{k-1} = |L^\sigma|$. This concludes the special case.

    Suppose that $\sigma$ contains a cycle of even length. Without loss of generality we assume that $s := s_r$ is even. We prove $|L^\sigma| = |R^\sigma|$ by induction on $r$. For the base case with $r = 1$, we have that $\sigma = (\sigma_1 \ \sigma_2 \ \dots \ \sigma_n)$ is an $n$-cycle where $n$ is even. It is easy to see that
    \[
        L^\sigma = \{\emptyset,\, [n]\} 
        \quad \text{and} \quad
        R^\sigma = \{\{\emptyset,[n]\},\, \{\sigma_1\sigma_3\dots\sigma_{n-1}, \sigma_2\sigma_4\dots\sigma_n\} \}.
    \]
    So we have $|L^\sigma| = 2 = |R^\sigma|$.

    For the induction step, assume that $r > 1$ and consider a permutation $\tau$ that has cycle type $s_1,s_2, \dots, s_{r-1}$. Without loss of generality, let us assume that $\tau$ is equal to the permutation $\sigma$ restricted to $[n-s_r]$. Define the set $S = [n] \setminus [n-s_r]$. It is easy to see that
    \[
        L^\sigma = L^\tau \sqcup \{A \cup S : A \in L^\tau\}
    \]
    and so we have $|L^\sigma| = 2|L^\tau|$. On the other hand, let us consider a partition $AB \in R^\tau$. If $\sigma(A) = A$, then we have that the partitions
    $(A \cup S)B$ and $A(B \cup S)$ lie in $R^\sigma$. On the other hand, if $\sigma(A) = B$, then write $(c_1, c_2, \dots, c_{s_r})$ for the cycle of $\sigma$ supported on $S$. Then we have
    \[
        (A \cup \{c_1, c_3, \dots, c_{s_k-1} \})(B \cup \{c_2, c_4, \dots, c_{s_k} \})
        \text{ and }
        (A \cup \{c_2, c_4, \dots, c_{s_k} \})(B \cup \{c_1, c_3, \dots, c_{s_k-1} \})
    \]
    are elements of $R^\sigma$. Every element of $R^\sigma$ arises uniquely in one of the ways described above. So it follows that $|R^\sigma| = 2|R^\tau|$. By induction, we have $|L^\tau| = |R^\tau|$ and so we deduce that $|L^\sigma| = |R^\sigma|$ and we are done.
\end{proof}

\begin{proof}[Proof of Theorem~\ref{thm: k=2 H* coeff}]
Fix $\sigma \in S_n$ with cycle type $(s_1, \dots, s_r)$ and denote by $C_1, \dots, C_r$ the cycle sets of $\sigma$ such that $|C_i| = s_i$ for each $i \in [r]$. Let us consider the coefficients given by Theorem~\ref{thm: H* coefficient in terms of representations} and Proposition~\ref{cor: subset for ell-h=1} for the second hypersimplex $\Delta_{2,n}$. We have
    \[
    H^*_m(\sigma) = |\Phi_2(\sigma, 2m)| - \lambda_1 |\Phi_1(\sigma, m-1)| = \begin{cases}
        |\Phi_2(\sigma, 2m)| & \text{if } m \neq 1, \\
        |\Phi_2(\sigma, 2m)| - \lambda_1 & \text{if } m = 1.
    \end{cases}
    \]
    The value $\lambda_1$ is equal to the number of fixed points of $\sigma$. So $\lambda_1(\sigma) = \rho_1 (\sigma)$ is the character of the natural representation of $S_n$. On the other hand, the set $\Phi_2(\sigma, 2m)$ contains all functions $f \colon [r] \rightarrow \{0,1\}$ such that $\sum_{i = 1}^r f(i) s_i = 2m$. There is a natural correspondence between $f \in \Phi_2(\sigma, 2m)$ and subsets $F \subseteq [n]$ with $|F| = 2m$ and $\sigma(F) = F$, which is given by
    \[
        f \mapsto F = \bigcup_{i \in f^{-1}(1)} C_i.
    \]
    So $|\Phi_2(\sigma, 2m)| = \rho_{2m}(\sigma)$ is equal to the permutation character of $S_n$ acting on $\binom{[n]}{2m}$. This proves that $H^*_0 = \Chi_0$, $H^*_1 = \rho_2 - \rho_1$, and $H^*_m = \rho_{2m}$ for each $2 \le m \le \lfloor n/2 \rfloor$. By Corollary~\ref{cor: degree of H* poly}, if $n$ is odd then the leading coefficient is $H^*_{(n-1)/2} = \rho_{n-1} = \rho_1$, otherwise if $n$ is even then the leading coefficient is $H^*_{n/2} = \rho_{n} = \Chi_0$.

    It remains to show that $H^*[1] = \Chi_0 + \tau_2 + \dots + \tau_{\lfloor n/2 \rfloor}$. If $n$ is odd, then the result follows from the above and the fact that $\rho_m = \tau_m$ for each $m \in [n]$. On the other hand, if $n$ is even, then result follows from Lemma~\ref{lem: equal reps for n even}.
\end{proof}

Theorem~\ref{thm: k=2 H* coeff} allows us to give a complete combinatorial proof of effectiveness for the $H^*$-polynomial.

\begin{corollary}\label{cor: k=2 H* is effective}
    Fix $n$. Each coefficient of $H^*(\Delta_{2,n}, S_n)$ is an effective representation.
\end{corollary}

\begin{proof}
    Let $m \in \{0,1,2, \dots, \lfloor n/2 \rfloor\}$ and consider the $t^m$ coefficient $H^*_m$ of $H^*(\Delta_{2,n}, S_n)$. If $m \neq 1$, then $H^*_m$ is a permutation character, hence it is effective. Otherwise, if $m = 1$ then let $V$ and $W$ be $\CC S_n$-modules with characters $\rho_1$ and $\rho_2$ respectively. Explicitly, we assume $V$ has basis $e_i$ with $i \in [n]$ and action $\sigma(e_i) = e_{\sigma(i)}$ and $W$ has basis $f_I$ with $I \in \binom{[n]}{2}$ and action $\sigma(f_I) = f_{\sigma(I)}$. Define the map
    $
    \varphi \colon V \rightarrow W$ given by 
    $
    \varphi (e_i) = \sum_{j \neq i} f_{ij}
    $. It is straightforward to show that $\varphi$ is an injective $\CC S_n$-module homomorphism, hence $H^*_1 = \rho_2 - \rho_1$ is effective.
\end{proof}

So, the above corollary shows that the coefficient $H^*_1$ is a representation. However, as we will now show, $H^*_1$ is special in that the trivial character does not appear as a direct summand.

\begin{corollary}\label{cor: k=2 H*_m is perm rep iff m not 1}
    Fix $n \ge 4$ and let $0 \le m \le \lfloor n/2 \rfloor$. Then the coefficient $H^*_m$ of the equivariant $H^*$-polynomial of $\Delta_{2,n}$ is a permutation character if and only if $m \neq 1$. Moreover, the trivial character does not appear in $H^*_1$.
\end{corollary}

\begin{proof}
    Suppose $m \neq 1$. By Theorem~\ref{thm: k=2 H* coeff}, we have that $H^*_m$ is the permutation character $\rho_{2m}$ if $m > 1$ and $\Chi_0$ if $m = 0$. Otherwise, let $m = 1$ and assume by contradiction that $H^*_1$ is a permutation character. For any permutation character $\rho$, a consequence of the Orbit-Stabiliser Theorem is that $\langle \Chi_0, \rho \rangle = \frac{1}{n!}\sum_{\sigma \in S_n} \rho(\sigma)$ is equal to the number of orbits of the action. By Theorem~\ref{thm: k=2 H* coeff}, we have $H^*_1 = \rho_2 - \rho_1$. Since $n \ge 4$, we have $H^*_1(1_{S_n}) = \binom n2 - n \neq 0$, hence $H^*_1 \neq 0$. Since $S_n$ acts transitively on $[n]$ and the $2$-subsets of $[n]$, the action associated to $H^*_1$ has $\langle \Chi_0, \rho_2 - \rho_1\rangle = \langle \Chi_0, \rho_2\rangle - \langle \Chi_0, \rho_1\rangle = 0$ orbits, a contradiction. This completes the proof.
\end{proof}

\begin{remark}
    Each coefficient of the $h^*$-polynomial of the hypersimplex has a combinatorial interpretation in terms of DOSPs. Explicitly $h^*_m$ is the number of hypersimplicial $(k,n)$-DOSPs with winding number $m$, see Definition~\ref{def: dist in dosp and winding no}. We note that the winding number is not invariant under the action of $S_n$, so the same interpretation does not apply in the most general setting. However, the winding number is invariant under the cyclic group $C_n \le S_n$. In \cite{elia2022techniques}, it is shown that the coefficient of $t^m$ in $H^*(\Delta_{k,n}; C_n)[t]$ is the number of $\sigma$-fixed hypersimplicial $(k,n)$-DOSPs with winding number $m$. In the case $k= 2$, this result can be deduced from Theorem~\ref{thm: k=2 H* coeff} as follows. If $D = ((A, 1), (B, 1))$ is a DOSP, then we define the set $J(D)$ of \emph{jumping points} to be the set of $i \in [n]$ such that $i$ and $i+1$ belong to different sets of $D$. Since $k = 2$, the winding number of $D$ is equal to half the number of jumping points. For $m = 0$ and $m \ge 2$, the restriction $\res^{S_n}_{C_n}\rho_{2m} (\sigma)$ counts the number of $\sigma$-fixed partitions $\{A, B\}$ of $[n]$ with $|A| = 2m$. For each such partition there is a unique DOSP with jumping points $A$. This DOSP is $\sigma$-fixed and has winding number $m$. Every such DOSP arises in this way and so $H^*_m$ is the number of $\sigma$-fixed DOSPs with winding number $m$. In the case $m = 1$, we have that $\res^{S_n}_{C_n} (\rho_2 - \rho_1)$ is isomorphic to the permutation representation of $C_n$ acting on the set of $2$-subsets $ij \in \binom{[n]}{2}$ such that $|i - j| > 1$. For each such $2$-subset, we obtain a $\sigma$-fixed hypersimplicial DOSP, which concludes the proof.
\end{remark}

\section{Discussion and open problems}\label{sec: discussion}
In this section, we discuss various open questions that arise from our investigation. 

\medskip \noindent \textbf{Formula for $H^*$ coefficients.} 
First, let us consider the formula for the number of $\sigma$-fixed hypersimplicial DOSPs in Theorem~\ref{thm: H*[1] is perm character of DOSPs under Sn}. Our proof is quite technical but leads to a significant amount of cancellation and simplification. So we ask whether there is a simpler or direct combinatorial proof of this result.

\medskip

Next, consider the formula for the coefficients of $H^*(\Delta_{k,n}; S_n)$ in Theorem~\ref{thm: H* coefficient in terms of representations}. By Proposition~\ref{prop: Phi set is perm rep}, the map $\sigma \mapsto |\Phi_{k-h}(\sigma, m(k-h)-h)|$ is a permutation character of $S_n$. However, for a given $h \in \{0,1, \dots, k-1\}$ and $I \in \mathcal I_h$ the map
\[
\sigma \mapsto \binom{\lambda_1}{I_1} \binom{\lambda_2}{I_2} \cdots \binom{\lambda_{k-1}}{I_{k-1}}
\]
is not generally the character of an $S_n$-representation, where $\lambda_i$ is the number of cycles of $\sigma$ of length $i$. If $k=2$ then the only non-trivial map is $\sigma \mapsto \lambda_1$, which is the character of the natural representation. In Section~\ref{sec: k=2}, we use this observation to give a simple description of the coefficients of the $H^*$-polynomial. So, for $k > 2$, we ask if the above map is the character of an $S_n$ representation. More generally we ask what are the irreducible representations appearing in $H^*_m$. 

\begin{question}
    For each $m \ge 0$, what is a multiplicity of each irreducible representation of $S_n$ in the coefficient $H^*_m$ of $H^*(\Delta_{k,n}; S_n)$?
\end{question}

\medskip
\noindent \textbf{Symmetric triangulations.}
Our next questions concern the well-known unimodular triangulation of the hypersimplex \cite{lam2007alcoved}. We have seen in Corollary~\ref{cor: k=2 H*_m is perm rep iff m not 1} that the trivial character does not appear in the coefficient of $t$ in $H^*(\Delta_{2,n}; S_n)$, for each $n \ge 4$. Therefore, these hypersimplicies provide a family of counterexamples to Conjecture~\ref{conj: stapledon positive coeff}; see Example~\ref{example: counterexample to stapledon positive coeff}. In \cite{stapledon2024invariant_triangulations}, Stapledon asks whether Conjecture~\ref{conj: stapledon positive coeff} holds if we assume that $P$ admits a $G$-invariant lattice triangulation. It would be interesting to know whether the hypersimplex supports this conjecture. In particular, we ask for which subgroups $G$ of $S_n$, the hypersimplex $\Delta_{k,n}$ admits a $G$-invariant lattice triangulation. And, in each of these cases, whether the coefficients of the $H^*$-polynomial are permutation representations.

\begin{example}\label{example: triangulation 2,4}
Consider the hypersimplex $\Delta_{2,4} \subseteq \RR^4$. We write $[I_1, I_2, I_3, I_4]$ for the convex hull of the points $\{e_{I_1}, e_{I_2}, e_{I_3}, e_{I_4}\} \subseteq \RR^4$, where $I_i$ is a $2$-subsets of $[4]$ for each $i \in [4]$. The well-known lattice triangulation of the hypersimplex is given by the simplices
\[
\nabla_1 = [12, 13, 14, 24],\quad
\nabla_2 = [23, 24, 12, 13],\quad
\nabla_3 = [34, 13, 23, 24],\quad
\nabla_4 = [14, 24, 34, 13].
\]
Let $a := (1 \ 2 \ 3 \ 4) \in S_n$ be a $4$-cycle. Observe that cyclic group $C_4$ generated by $a$ acts on the set of simplices by permutation. For instance, we have
\[
a(\nabla_1) := [a(12), a(13), a(14), a(24)] = [23, 24, 12, 13] = \nabla_2
\]
and, in general, we have $a(\nabla_i) = \nabla_{a(i)}$ for each $i \in [4]$. By Theorem~\ref{thm: hypersimplex cyclic action}, the coefficients of the $H^*$-polynomial $H^*(\Delta_{2,4}; C_4)$ are  permutation representations, hence they satisfy Conjecture~\ref{conj: stapledon positive coeff}. 

Observe that the simplices are also invariant under the transposition $b := (1 \ 3) \in S_4$. Explicitly, we have $b(\nabla_i) = \nabla_{b(i)}$ for all $i \in [4]$. So, the dihedral group $D_8 := \langle a, b \rangle \le S_4$ acts naturally on the lattice triangulation. By Theorem~\ref{thm: k=2 H* coeff}, to show that Conjecture~\ref{conj: stapledon positive coeff} holds for $H^*(\Delta_{2,4}; D_8)$, it suffices to consider the coefficient of $t$, which is given by the character $H^*_1 = \rho_2 - \rho_1$. One easily verifies this character is the permutation character of $D_8$ acting on the set $\{13, 24\}$, hence Conjecture~\ref{conj: stapledon positive coeff} holds. We also observe that the set of simplices is not invariant under the transposition $(1 \ 2) \in S_n$. Therefore, the subgroup $D_8$ is the maximal subgroup of $S_4$ that leaves the triangulation invariant.
\end{example}

Our computations suggest that the above example generalises to all hypersimplices. For each hypersimplex $\Delta_{k,n}$ with $(k,n)$ in the set
\[
\{(2,4),\, (2,5),\, \dots,\, (2,10),\, (3,6),\, (3,7),\, (3,8)\}
\]
we have verified that the symmetry group of the well-known triangulation is the dihedral group $D_{2n} \le S_n$. So, we make the following conjecture.

\begin{conjecture}\label{conj: triangulation under dihedral group}
    For all $0 < k < n$, the well-known lattice triangulation of $\Delta_{k,n}$ is invariant under the dihedral group $\langle a, b\rangle = D_{2n} \le S_n$ generated by the permutations
    \[
    a = (1 \ 2 \ \cdots \ n) 
    \quad \text{and} \quad
    b = \begin{cases}
        (1 \ n)(2 \ n-1)(3 \ n-2) \cdots ((n-1)/2 \ (n+3)/2) & \text{if } n \text{ is odd},\\
        (1 \ n)(2 \ n-1)(3 \ n-2) \cdots (n/2 \ n/2+1) & \text{if } n \text{ is even}.
    \end{cases}
    \]
    Moreover, $D_{2n}$ is the maximal subgroup of $S_n$ for which the lattice triangulation is invariant.
\end{conjecture}

We note that by Theorem~\ref{thm: k=2 H* coeff}, the coefficient $H^*_1$ of $t$ in $H^*(\Delta_{2,n}; D_{2n})$ is a permutation character. Explicitly, it is straightforward to show that $H^*_1$ is the permutation character of $D_{2n}$ acting on the set $\{\{i,j\} \subseteq [n] : |i-j| \ge 2 \}$, which can be thought of as the set of diagonals of an $n$-gon. It would be interesting to know if this generalises to all hypersimplices $\Delta_{k,n}$ with $k > 2$. 

\begin{question}
    Let $k < n$. For which groups $G \le S_n$ does there exist a $G$-invariant lattice triangulation of $\Delta_{k,n}$? Given such $G$, does Conjecture~\ref{conj: stapledon positive coeff} hold for $H^*(\Delta_{k,n}; G)$? Does Theorem~\ref{thm: hypersimplex cyclic action} hold if we replace the cyclic group $C_n$ with the dihedral group $D_{2n}$ (as defined in Conjecture~\ref{conj: triangulation under dihedral group})?
\end{question}

We remark that the final question is well defined as the winding number of a DOSP is invariant under the action of the dihedral group.
We also note that $\Delta_{2,n}$ does not admit a $S_n$-invariant lattice triangulation for any $n \ge 4$. To see this, we write $e_{ij} := e_i + e_j \in \RR^n$ for each $i,j \in [n]$ and observe that any triangulation of $\Delta_{2,n}$ must include a simplex $\nabla = \conv\{e_{12}, e_{13},\ldots,e_{1n},e_{ab}\}$ with $2 \le a < b \le n$.
This follows from the fact that $\Delta_{2,n} \cap \{x \in \RR^n : x_1 = 1\}$ is a facet of $\Delta_{2,n}$ given by the convex hull of the vertices $e_{1j}$ with $2 \le j \le n$.
Since $n \ge 4$, there exists $c \in [n] \setminus \{1,a,b\}$. Let $\sigma$ be the $3$-cycle $(a \ b \ c) \in S_n$. The relative interior of the simplex $\sigma(\nabla) = \conv\{e_{12}, \dots, e_{1n}, e_{bc}\}$ intersects the relative interior of $\nabla$, hence they cannot belong to the same triangulation. So the original triangulation is not $S_n$-invariant.

\section*{Acknowledgements}

The authors would like to thank Nick Early for explaining the motivation behind DOSPs and the connection to equivariant volumes of hypersimplices.
The second author was supported by a Grant-in-Aid for JSPS Research Fellows (Grant Number: 24KJ1590) under the JSPS DC2 programme.

\bibliographystyle{plain}
\bibliography{bibliography}

\bigskip
\bigskip
\noindent 
\footnotesize

\noindent
{\sc Oliver Clarke, Department of Mathematical Sciences, Durham University}\\ 
\textit{Email address}: \texttt{oliver.clarke@durham.ac.uk}\\
\textit{URL}: \texttt{https://www.oliverclarkemath.com/}\\

\noindent
{\sc Max K\"olbl, Department of Pure and Applied Mathematics, Osaka University}\\ 
\textit{Email address}: \texttt{max.koelbl@ist.osaka-u.ac.jp}\\

\end{document}